\documentclass[11pt,a4paper]{amsart}

\usepackage{fullpage}
\usepackage{amsmath}
\usepackage{amsfonts}
\usepackage{amssymb}
\usepackage{amsthm}
\usepackage[utf8]{inputenc}
\usepackage{breqn}
\usepackage{afterpage}
\usepackage{longtable}
\usepackage{indentfirst}
\usepackage{caption}
\usepackage{subcaption}
\usepackage{graphicx}
\graphicspath{ {./images/} }
\newtheorem{theorem}{Theorem}[section]
\newtheorem{lemma}[theorem]{Lemma}
\newtheorem{corollary}[theorem]{Corollary}
\newtheorem{proposition}[theorem]{Proposition}

\theoremstyle{definition}
\newtheorem{definition}[theorem]{Definition}

\newtheorem{theorem-definition}[theorem]{Theorem-Definition}

\theoremstyle{remark}
\newtheorem{remark}[theorem]{Remark}

\numberwithin{equation}{section}

\usepackage{color}
\usepackage{xcolor}

\newcommand{\lam}{\lambda}
\newcommand{\Pk}{\mathbb{P}^k}
\newcommand{\C}{\mathbb{C}}

\renewcommand{\P}{\mathbb{P}}
\DeclareMathOperator{\jac}{Jac}
\DeclareMathOperator{\Supp}{Supp}
\DeclareMathOperator{\diam}{diam}

\DeclareMathOperator{\Lip}{Lip}

\renewcommand{\epsilon}{\varepsilon}

\begin{document}

\title[Strong probabilistic
stability in holomorphic
families]{Strong probabilistic
stability in holomorphic families
of endomorphisms of $\Pk  (\mathbb C)$ and polynomial-like maps}

\author{Fabrizio Bianchi}
\author{Karim Rakhimov}

\address{CNRS, Univ. Lille, UMR 8524 - Laboratoire Paul Painleve, F-59000 Lille, France}
\email{fabrizio.bianchi@univ-lille.fr}
\address{CNRS, Univ. Lille, UMR 8524 - Laboratoire Paul Painleve, F-59000 Lille, France}

\curraddr{National University of Singapore, Lower Kent Ridge Road 10,
Singapore 119076, Singapore}

\email{rkarim@nus.edu.sg, karimjon1705@gmail.com}

\subjclass[2010]{}
\date{}


\begin{abstract}
We prove that, in 
stable
families
of endomorphisms of $\mathbb{P}^k (\mathbb C)$, all invariant measures
whose measure-theoretic entropy is strictly larger than $(k-1)\log d$
at a given 
parameter
can be followed holomorphically with the parameter  in all the parameter space.
As a consequence, almost all points (with respect
to any such measure at any parameter)
in the Julia set can be followed holomorphically
without intersections.
This generalizes previous results
 by Berteloot, Dupont, and the first author 
for the measure of maximal entropy,
and
provides
a parallel in this setting to the probabilistic stability of
H\'enon maps by Berger-Dujardin-Lyubich.
 Our proof relies both on  techniques from the theory of stability/bifurcation in any dimension
and on an explicit lower bound for the Lyapunov exponents for an ergodic
measure in terms of its measure-theoretic entropy, due to de Th{\'e}lin and Dupont.

 A local version of our result holds also for all measures supported on the Julia set with just strictly positive Lyapunov exponents and not charging the post-critical set. Analogous results hold in  families
of polynomial-like maps 
 of large topological degree.
In this case, 
 as part of our proof, we
also  give 
a sufficient condition for the positivity
of the Lyapunov exponents 
of an ergodic measure
for a polynomial-like map in any dimension
in term of its  measure-theoretic entropy,
generalizing to this setting the analogous result by de Thélin and Dupont valid on $\mathbb P^k (\mathbb C)$.
\end{abstract}

\maketitle

\subsection*{Notation}

We denote
by $\mathbb{P}^k=  \mathbb P^k (\mathbb C)$
the complex projective space of dimension $k$. 
A  $(p,p)$-current is a current of bidegree $(p,p)$.
If $W \subset \mathbb C^k$ is an open set
and
$S$  a positive closed $(p,p)$-current on $W$,
we denote
its mass
by  $\|S\|_W:=\int_W S\wedge \omega^{k-p}$,
where $\omega$ is the standard K\"{a}hler form of $\mathbb{C}^k$.
Given a holomorphic map $f$ we denote by $\jac f$ 
the Jacobian of $f$. If $\nu$ is an ergodic $f$-invariant
probability measure we denote by $h_\nu(f)$ the measure-theoretic entropy
of $\nu$ with respect to $f$.

\section{Introduction}

By a classical result by Mané-Sad-Sullivan \cite{MSS83}
and Lyubich \cite{Ly83a}, the stability of a family of rational
maps
of degree $d\geq 2$
is determined by the stability of its repelling cycles.
More precisely, the so-called $\lambda$-lemma ensures that,
once one can follow holomorphically with the parameter a dense
subset of the Julia set $J_{\lam_0}$
(the support of the unique measure of maximal entropy of $f_{\lam_0}$ \cite{FLM,Ly83b})
at
a given parameter $\lambda_0$,
then \emph{every} point of $J_{\lam_0}$
can be followed holomorphically.
The important point here is that  the motions corresponding
to different points (that exist thanks to Montel theorem)
do not intersect. This is a consequence of Hurwitz theorem.
The \emph{stability locus}
is the (open dense)
subset of the parameter space where
the above condition of stability holds true. The complement
of such set is the \emph{bifurcation locus}. By a result of
 DeMarco \cite{dM03}, such locus is the support of
 a naturally defined \emph{bifurcation current}, see
 also \cite{Prz85,Si81} for the polynomial case.

A generalization of the theory by Mané-Sad-Sullivan, Lyubich,  and DeMarco
to 
 families of 
endomorphisms of $\mathbb P^k$
 of algebraic degree $d \geq 2$
in any dimension $k\geq 1$ has been recently
carried out by Berteloot, Dupont, and the
first author in \cite{BBD18,B19}, see also the presentation
in \cite{BB18}.
 As most of the one-dimensional techniques do not apply anymore,
the main tools were ergodic
and (pluri)potential theory. Very roughly
speaking, compactness of suitable
spaces of currents and plurisubharmonic functions
played the role of Montel theorem.
And precise statistical properties of the measure
of maximal entropy
\cite{BD01,DS10,FS94}
(among them,
an explicit lower bound for its Lyapunov exponents  \cite{BD99})
played somehow the role of an asymptotic Hurwitz theorem.
As a consequence, dynamical stability
in such families
 (defined for instance by the vanishing of a natural bifurcation current \cite{BB07, Pha05}, or by a condition on the stability of the repelling points)
is equivalent to
the existence of a holomorphic motion for
a \emph{full measure subset} of the Julia set,
with respect to the 
 measure of maximal entropy.
We refer to \cite{AB18,B18,BB18h,BB22,Duj22} for further developments  of the theory of stability/bifurcation in any dimension, and in particular to 
\cite{AB22,BT17,Bie19,Duj17,GTV23,Taf21} for new phenomena with respect to the one-dimensional case.

The main goal of this paper is
 to strengthen the main result in \cite{BBD18} by showing that,
in stable families of endomorphisms of $\mathbb P^k$
 of algebraic degree $d\geq 2$,
dynamical stability implies the existence of a well-defined
 local
motion for \emph{almost all points} with respect
to \emph{all measures 
 on the Julia set
with strictly positive Lyapunov
exponents and not charging the post-critical set}. By a result of de Th{\'e}lin and Dupont
\cite{dTh08,D12},
this
in particular
 applies to all ergodic measures whose
  measure-theoretic entropy is  strictly
 larger
than $(k-1) \log d$.
In this case, the motion is well-defined on all the
parameter space.
Observe that the topological entropy of an endomorphism
of $\mathbb P^k$ of algebraic degree $d$
is equal to $k \log d$ \cite{G03} (which is
then equal to the  measure-theoretic 
entropy of the unique measure of
maximal entropy), and that $(k-1)\log d$ is a natural
threshold for many dynamical phenomena.
See \cite{BD20,BD22,D12,SUZ,UZ13}
for large classes of examples of such measures
in any dimension,
and their statistical properties.

Let us notice that an analogous result is already known
in the setting of polynomial diffeomorphisms of $\C^2$
(usually called H\'enon maps). Indeed,
a parallel theory to that of \cite{BBD18,B19}
has been developed  in this setting
by Dujardin and Lyubich, see \cite{DL15, Duj22}. Stability is
defined also in this setting by means of
a number of equivalent conditions, among them
the existence of a \emph{branched} holomorphic motion
for the Julia sets, meaning that natural motions
of distinct points can a priori intersect.
In \cite{BD17} Berger and Dujardin
proved that such motion is \emph{unbranched}
at almost every point
with respect to all measures of positive entropy.
Since the topological entropy is the logarithm of the algebraic
degree in this context, this corresponds to our threshold.
 The current work was inspired by \cite{BD17} and provides a parallel to that result for families of endomorphisms of $\mathbb P^k$. As was already the case for \cite{DL15} and \cite{BBD18}, our approach is completely different from the one in \cite{BD17}, since the key point here is to understand the relation between the Julia sets and the dynamics of the critical set (which does not exist in the case of diffeomorphims of $\mathbb C^2$). This is achieved with techniques from pluripotential theory.

\subsection{Definitions and results}
We will consider in this paper a more general
setting than that of  the endomorphisms of $\mathbb P^k$,
that is that of \emph{polynomial-like maps
of 
 large
topological degree},
see Definition \ref{def:large-top-deg}.
However, we restrict to the family $\mathcal H_d (\P^k)$
of all the endomorphisms of $\P^k$ of a given
algebraic degree $d  \geq 2$
in this introduction for simplicity.

The main result of \cite{BBD18} in this setting
can be stated as follows.
 Given a holomorphic family $(f_{\lam})_{\lam \in M}$ of endomorphisms of $\mathbb P^k$,
we denote by $\mu_\lam$
the equilibrium measure of $f_\lam$
(i.e., the unique measure of maximal entropy $ k \log d$ of $f_\lambda$
\cite{BD01,DS10,FS94}),
and recall that the Julia set $J_\lambda$
of $f_\lam$ is by definition the support of $\mu_\lam$.

\begin{theorem-definition}[Berteloot-B.-Dupont \cite{BBD18}]\label{t:BBD-intro}
Let $M$ be an open connected and simply connected subset
of the family $\mathcal H_d (\P^k)$ of the
endomorphisms of $\Pk$
of a given algebraic degree $d  \geq 2$. The
following conditions are equivalent:
\begin{enumerate}
    \item the repelling points in the Julia sets move 
    holomorphically with $\lambda$ (see Definition 
    \ref{def:motion-rep-all});
    \item the sum $L(\lam)=\int \log |\jac f_\lambda|\mu_\lambda$ of the Lyapunov exponents of 
    $\mu_\lam$ satisfies $dd^c L \equiv 0$;
    \item there exists an equilibrium lamination.
\end{enumerate}
We say that a family is \emph{stable}
if any of the
equivalent
conditions above
holds.
\end{theorem-definition}

An equilibrium lamination is defined as follows.
 Denote by $\mathcal J$ the set of all
 holomorphic maps
 $\gamma\colon M\to \P^k$ such that $\gamma (\lam)$
 belongs to $J_\lam$ for all $\lam\in M$.
 We often  identify a map $\gamma$ with its graph
$\Gamma_\gamma$ in the product space $M\times \P^k$.
The family $(f_\lam)_{\lam \in M}$ induces a dynamical
system $\mathcal F$ on the space $\mathcal J$
by $\mathcal F \gamma (\lam) := f_\lam (\gamma(\lam))$.
Observe also that the maps $f_\lam$ can be seen as fibers
of a single holomorphic map
$f\colon M\times \P^k \to M\times \P^k$.
We denote by $C_f$ the critical set of
such $f$, and by $GO(C_f)$
the grand orbit of $C_f$, i.e.,
$GO (f) := \cup_{n,m\geq 0} f^{-m} ( f^{n} (C_f))$.
We also denote by $d_t$ the topological degree of any element of $\mathcal H_d (\mathbb P^k)$. Observe that in this case we have $d_t = d^k$.

\begin{definition}\label{def:lamination}
A \emph{dynamical lamination}
for the family
$f$ is an $\mathcal{F}$-invariant subset $\mathcal{L}$
of $\mathcal{J}$ such that
\begin{enumerate}
    \item $\Gamma_\gamma\cap \Gamma_{\gamma'}=\emptyset$
for all    $\gamma\ne \gamma'\in\mathcal{L}$;
    \item $\Gamma_\gamma\cap GO(C_f)=\emptyset$
 for all   $\gamma\in\mathcal{L}$;
    \item $\mathcal{F}:\mathcal{L}\to \mathcal{L}$ is $d_t$-to-$1$.
\end{enumerate}

An \emph{equilibrium lamination} or
$\mu_\lambda$-\emph{measurable holomorphic motion of $J_\lambda$}
is a dynamical
lamination satisfying the
following further property:
\begin{enumerate}
\setcounter{enumi}{3}
    \item $\mu_{\lam}(\{\gamma(\lambda):\gamma\in\mathcal{L}\})=1$
    for all $\lam\in M$;
\end{enumerate}
\end{definition}

Our main result  in the case of the family $\mathcal H_d(\P^k)$ can be stated as follows.

\begin{theorem}\label{t:main-intro}
Let $M$ be an open  connected and simply connected subset
of the family $\mathcal H_d (\P^k)$ of the
endomorphisms of $\Pk$
of a given algebraic degree $d  \geq 2$
and assume that the family $(f_\lam)_{\lam \in M}$
is stable.
Then there exists a dynamical
lamination $\mathcal L$
satisfying
\begin{itemize}
    \item[(4')] $\nu (\{\gamma(\lambda):\gamma\in\mathcal{L}\})=1$
for every $\lam \in M$ and
every ergodic $f_\lam$-invariant measure $\nu$
whose 
 measure-theoretic entropy is strictly larger than $(k-1)\log d$.
\end{itemize}
\end{theorem}

In particular, given any $\lam_0\in M$ and any ergodic
$f_{\lam_0}$-invariant  measure 
 $\nu_0$
whose
 measure-theoretic entropy is strictly larger than $(k-1)\log d$, it is possible
to define a measurable holomorphic motion of $J_\lam$
associated to the measure
 $\nu_0$
 on all of $M$.

\smallskip

An analogous result holds in the much more general
setting of families of polynomial
like maps of 
 large
topological degree
(see Definition \ref{def:large-top-deg}),
or of arbitrary subfamilies
of $\mathcal H_d (\P^k)$.
A version of Theorem-Definition \ref{t:BBD-intro}
in this more general setting is
the main result of \cite{B19}, see also \cite{B16}.
We refer to 
Theorem 
\ref{t:mainlamin} and  Corollary \ref{c:pk}
for 
precise statements of our results in these cases
 and to
 Theorem
\ref{t:mainequiv}
for a weaker (local)
version of it holding for all measures supported on the Julia set with strictly positive Lyapunov exponents and 
not charging the postcritical set.
As part of our proof,
we also prove a generalization of de Thélin and Dupont
theorem above 
\cite{dTh08,D12} in this setting, giving the strict positivity
of the Lyapunov exponents of measures of sufficiently
large
 measure-theoretic entropy
for polynomial like maps of 
 large
topological degree,
see Theorem \ref{t:lyapunov-positive}.
This result is crucial to get a uniform domain of existence for the laminations associated with different measures, depending only on their measure-theoretic entropy.

\subsection{Organization of the paper}
In Section \ref{s:PLmaps} we recall some definitions and
results on polynomial-like maps 
 of large
topological degrees.
We also study
the Lyapunov exponents of measures 
with sufficiently
large  measure-theoretic entropy
and prove their strict positivity,
 with an explicit 
lower bound depending on the measure-theoretic entropy of the measure,
see Theorem \ref{t:lyapunov-positive}.
In Section \ref{s:fplmaps} 
we recall definitions and results
 on holomorphic families 
 of polynomial-like maps 
and we state the precise versions of our main results, see
 Theorems \ref{t:mainequiv}  and  \ref{t:mainlamin} and Corollary \ref{c:pk}.
Section \ref{s:proof-thm} is dedicated to the proof
of 
 these results.
In Appendix \ref{app:contraction} we 
record an intermediate contraction estimate along generic inverse orbits
(with respect to naturally defined measures) in the space $(\mathcal J, \mathcal F)$,
for later reference.

\subsection*{Acknowledgments}
The authors would like to thank Romain Dujardin, whose original question motivated the work on this problem, and Maxence Br{\'e}vard, Viet-Anh Nguyen, and Johan Taflin for useful comments and discussions.

This project has received funding from
 the French government through the Programme
 Investissement d'Avenir
 (I-SITE ULNE /ANR-16-IDEX-0004,
 LabEx CEMPI /ANR-11-LABX-0007-01,
ANR QuaSiDy /ANR-21-CE40-0016,
ANR PADAWAN /ANR-21-CE40-0012-01)
managed by the Agence Nationale de la Recherche.

\section{Polynomial-like maps}\label{s:PLmaps}

\subsection{Preliminary notions}

We recall here the main notions and results
about polynomial-like maps that we will use in the sequel.
We refer to \cite{DS03,DS10}
for more details.

\begin{definition}
 A \emph{polynomial-like map} is a proper holomorphic map
  $f:U\to V$, where $U \Subset V$ are open subsets of
  $\C^k$ and $V$ is convex.
  \end{definition}
  
 In particular, a
 polynomial-like map $f$ is a ramified covering $U\to V$,
 and the \emph{topological degree} $d_t$
 (i.e., the number of preimages of any point in $V$, counting the multiplicity)
 of $f$ is well-defined.
 We will always assume that
 $d_t\ge 2$. The (compact)
 set $K:= \bigcap_{n=1}^{\infty} f^{-n}(U)$ is the
 \emph{filled-in Julia set} of $f$.
It consists of the set of points whose orbit is well-defined.
The system $(K,f)$ is a well-defined
dynamical system.

Notice in particular that homogeneous
lifts to $\C^{k+1}$ of endomorphisms of $\P^k$
give polynomial-like maps.
On the other hand, while in dimension
$k=1$ any polynomial-like map
is conjugated to an actual
polynomial on the Julia set \cite{DH85},
in higher dimensions this class
is known to be much larger than
that of regular polynomial endomorphisms of
$\P^k$ (i.e., those extending holomorphically to $\P^k$),
see for instance \cite[Example 2.25]{DS10}.

 \begin{definition}\label{def:degrees}
    Let $f\colon U\to V$ be a polynomial-like map.
    For $0\le p\le k$ we set
    $$d_p=d_p(f)
    :=\limsup_{n\to\infty}\|(f^n)_*(\omega^{k-p})\|_W^{1/n}$$
    and
    \begin{equation}\label{stardegree}
    d^*_p=d^*_p(f)
    :=\limsup_{n\to\infty}\sup_S\|(f^n)_*(S)\|_W^{1/n}
    \end{equation}
       where $W\Subset V$ is a neighbourhood of $K$,
       $\omega$ is the standard K\"ahler form on $\mathbb C^k$,
       and the supremum in \eqref{stardegree} is taken
       over all positive closed $(k-p,k-p)$-currents of mass
       less than or equal to 1 on a fixed neighbourhood
       $W'\Subset V$
        of $K$. We say
       that $d_p$ and $d^*_p$ are the \emph{dynamical}
       and \emph{$*$-dynamical degrees of order $p$} of $f$, 
       respectively.
       Similarly, we define
          \begin{equation}\label{deltadegree}
    \delta_p=\delta_p(f)
    :=\limsup_{n\to\infty}\sup_X\|(f^n)_*[X]\|_W^{1/n}
    \end{equation}
    where the supremum in \eqref{deltadegree}
    is taken over all $p$-dimensional complex  analytic
    sets $X\subset V$.
  \end{definition}
  
 These definitions
  are independent of the
  open neighbourhoods
  $W,W'$.
Moreover, we have
 $d_p\le d^*_p$
  and
 $ \delta_p\le d^*_p$
  for all $0\leq p\leq k$,
$d^*_0=1$, and $d_k=\delta_k=d^*_k=d_t$.
 In the case of
 endomorphisms of $\P^k$
 of algebraic degree $d$,
  for every $ 0\leq p\leq k$
 the above
 definitions reduce to $d^*_{p} =\delta_p= d_p= d^p$.
 The next lemma
 in particular implies that
  $d_p\le \delta_p$.

\begin{lemma}[{see \cite[Lemma 4.7]{B19}}]\label{l:thesislemmaA}
Let $f\colon U\to V$ be a polynomial-like map.
Let $X$ be an analytic subset of $V$ of pure dimension  $p\le k-1$.
There exists a function $C\colon \mathbb N\to \mathbb R$ with $\limsup_{n\to\infty}C(n)^{1/n}=1$ and independent of $X$ such that
\begin{equation}\label{Xdeltap}
    \int_{{f^{-n}(V)}}[X]
    \wedge (f^{n_1})^*\omega\wedge
    \ldots\wedge (f^{n_{p}})^*\omega\le C(n)\,{\delta_p^n}
\end{equation}
and
\begin{equation}\label{omegadeltap}
    \int_{ f^{-n}(V)}
    \omega^{k-p}\wedge (f^{n_1})^*\omega\wedge\ldots\wedge (f^{n_{p}})^*\omega\le C(n) \,{\delta_p^n}
\end{equation}
where $0\le n_j\le n-1$ for all $1\le j\le p.$
In particular, the topological entropy
of the restriction of $f$ to $X$ is bounded by $\log \delta_p$.
\end{lemma}

\begin{proof}
The proof of \eqref{Xdeltap}
follows the strategy used by Gromov
to estimate
the topological entropy of endomorphisms of
$\mathbb P^k$, see for instance \cite{G03},
and adapted by Dinh and Sibony
\cite{DS03,DS10}
to the setting of
polynomial-like maps.
Since only minor modifications are needed, we
refer to
\cite[Lemma A.2.6]{B16} for a complete proof.
The second expression is deduced from the first
(by possibly increasing $C(n)$ by a bounded factor)
as $\omega^{k-p}$ can be written as an
average of currents of integration on $p$-dimensional
analytic sets.
\end{proof}

\begin{corollary}[{see \cite[Lemma A.2.9]{B16}}]
\label{corol-mass-0} 
Let $f\colon U\to V$ be a polynomial-like map.
Let $\nu$ be an  ergodic
 $f$-invariant 
  probability measure such
that $h_{\nu}  (f) > \log \delta_p$. Then $\nu$ gives
no mass to analytic sets of dimension 
less than or equal to $p$.
\end{corollary}

\begin{proof}
Let $X$ be an analytic set of dimension  at most $p$,
and assume by contradiction that $m:=\nu (X)>0$.
We can assume that $X$ is irreducible and
has pure dimension $q\leq p$. Since $\nu$ is invariant,
for all $n\in \mathbb N$ we have
$\nu (f^n (X))= \nu (f^{-n} (f^n (X)))\geq \nu(X)=m$.
Hence, we must have $\nu (f^{n_1} (X)\cap f^{n_2} (X))>0$
for some $n_1\neq n_2 \in \mathbb N$.
By the minimality of $X$, we deduce that
$f^{n_1} (X)= f^{n_2} (X)$.
Up to replacing $f$ with an iterate,
we can then assume that $X$ is
invariant and such that $\nu (X)>0$.
The ergodicity of $\nu$
implies that $\nu(X)=1$.
This implies that $h_\nu  (f)$ is smaller than
 or equal to the
topological entropy of the restriction of $f$ to $X$.
Since this is bounded by $\log \delta_p$ by
Lemma \ref{l:thesislemmaA}, the assertion follows.
\end{proof}

 We will focus in this paper
 on maps satisfying the following condition \cite{DS10}.
 Observe that the condition is always satisfied by lifts
 of endomorphisms, and moreover it is stable by small
 perturbation of the coefficients
 (since $d^*_{k-1}(f)$ depends upper semicontinuously from the map $f$).
This gives large families of examples.

 \begin{definition}\label{def:large-top-deg}
    We say that a polynomial-like map $f$
   has  
    \emph{large
   topological degree}
   if 
     $d^*_{k-1}<d_t$.
 \end{definition}

  Polynomial-like
 maps of large topological degree
enjoy many of the dynamical properties
of endomorphisms
(however, their study is usually technically
more involved, because of the lack of a naturally defined
Green function).
For instance, they admit a unique measure of maximal
entropy $\log d_t$, whose Lyapunov exponents (see below)
are strictly positive.
We denote this measure by $\mu$ and
define the Julia set as the support of $f$. Observe that
$J$ is a subset of the boundary of $K$.

\smallskip

In the current paper, we will often need that $f$ satisfies the (a priori) stronger property that
 $\max\{d_0^*,\ldots,d_{k-1}^*\}<d_t$. The following result by Dinh and the authors
 implies that this is always the case for maps of large topological degree.
 
\begin{proposition}[{see  \cite[Theorem 1.3]{BDR23}}]
\label{prop:BDR}
    Let $f$ be a polynomial-like map.
    Then the sequence $\{d_j^*\}_{0\leq j \leq k}$ is non-decreasing. In particular, if $f$
    has large topological degree then \[\max\{d_0^*,\ldots,d_{k-1}^*\} = d^*_{k-1} <d_t.\]
\end{proposition}

\subsection{Strict positivity of the Lyapunov exponents}
Let $f\colon U\to V$ be a polynomial-like map and
$\nu$ 
an ergodic $f$-invariant 
 probability
measure.
By Oseledets 
Theorem \cite{O68},
 $\nu$  admits $k$ \emph{Lyapunov exponents}
(counting
multiplicity, and admitting the possible value of $-\infty$).
We will denote them by
$$-\infty\le\chi_l< \chi_{l-1}<\ldots< \chi_1,$$
with multiplicity $m_j\geq 1$, $j=1,\ldots,l$ respectively.
Note that $m_1+m_2+\ldots+m_l=k$. As soon as
the 
 measure-theoretic
entropy of $\nu$ is positive,
 its
largest Lyapunov exponent
 is strictly positive by Ruelle inequality \cite{Ru78}
(the result is stated for compact manifolds, but
the proof applies here too,
by triangulating a neighbourhood of the filled Julia set $K$).

The following is the  main result of this section.

\begin{theorem}\label{t:lyapunov-positive}
Let  $f\colon U\to V$
be a polynomial-like
map of 
 large
topological degree.
 Let $\nu$ be an ergodic $f$-invariant
 probability
measure satisfying $h_\nu (f)>\log d^*_{ k-1} (f)$.
Then
\begin{enumerate}
\item\label{i:2k}
all the Lyapunov exponents of $f$
with respect to $\nu$
are
larger than or equal to
$(h_{\nu} (f)-\log d^*_{ k-1} )/(2m_l)>0$
 (where $m_l$ is the multiplicity of the smallest Lyapunov exponent of $\nu$);
\item
the function $\log|\jac f|$ is integrable
with respect to $\nu$.
\end{enumerate}
\end{theorem}

The second assertion
follows from the first, since
the sum of the Lyapunov exponents is equal to
$\langle \nu, \log |\jac f|\rangle$ 
 by Birkhoff Theorem, 
so
this last integral
is finite as soon as all Lyapunov exponents are finite.
Hence, we only need to
prove the first assertion. This is
known in the case of
endomorphisms of
$\Pk$, see
  de Thélin \cite{dTh08}
and Dupont \cite{D12}.
When $\nu$ is the measure of
maximal entropy,
it is a result of Dinh-Sibony \cite{DS03}, see \cite{BD99}
for the case of endomorphisms of $\mathbb P^k$.
Although we will follow the strategy of the proof
de Thélin and Dupont,
 because of the lack of a Hodge theory in this setting,
we will need to replace some cohomological arguments
when working with polynomial-like maps.
 We also cannot exploit the linearity of the sequence $\{d^{*}_p\}_{1\leq p\leq k}$  of the dynamical degrees, see also Remark \ref{r:logcon}.

\smallskip

As in \cite{D12},
the assertion
will
be deduced from the following
estimate. Recall that $l\geq 1$ is the number of distinct Lyapunov exponents of $\nu$.

\begin{theorem}\label{theoremD}
Let $f\colon U\to V$
be a polynomial-like
map and let $\nu$ be an ergodic $f$-invariant
 probability
measure.
Then,
if $l\ge 2$,
for every $2\le j\le l$, we have
\begin{equation}\label{inequalityD}
    h_\nu(f)\le \log \max_{1\le i\le k-l_j} \delta_{i}+2m_j\chi^+_j+\ldots+2m_l\chi^+_l
\end{equation}
  where we set $\chi^+:=\max\{\chi,0\}$, and $l_j=m_j+\ldots +m_l$.
\end{theorem}

If $h_{\nu} (f)=0$,
the assertion is clear. So we can assume
that $h_{\nu} (f)>0$. As we mentioned above we have
$\chi_1 >0$ by Ruelle inequality. Let  $1\le q\le l$ be such that
$q:=\max\{i: \chi_i>0\}$. Since the same proof below works for
the case  $l=q$, without loss of generality we
can assume that $l>q$. It is not difficult to see that the
 right-hand side of \eqref{inequalityD}, seen as a function of $j$,
 is non-decreasing
 for $j\ge q+1$, hence
 it is enough to prove Theorem \ref{theoremD}
 for $2\le j\le q+1$.

\begin{proposition}[{see \cite[Proposition 6.3]{D12}}]\label{propositiondupond} Let $f\colon U\to V$
be a polynomial-like map.
Fix $0<\beta_0\le1$. For every $\varepsilon>0$ there exists
$r_0>0$ and $n_0\in \mathbb N$ and, for any $n\ge n_0$,
a maximal $(n, r_0)$-separated  set $\mathcal{E}_n$
with $\mathrm{Card} (\mathcal{E}_n)\ge  e^{n(h_\nu(f)-2\varepsilon)}$
and such that, for any $z\in \mathcal{E}_n$ and
every $2\le j\le q+1$, there exists a neighbourhood $U_z^j$
of the origin
  of $\mathbb{D}^{l_j}$ 
and an injective mapping
  $\Psi_z^j: U_z^j\to f^{-n}(V)$  which satisfies the
  following properties:
    \begin{enumerate}
    \item $\Psi_z^j(0)=z$ and $\mathrm{Lip} \Psi_z^j\le \beta_0$;
    \item $\diam f^i(\Psi^j_z(U_z^j))\le e^{-n\varepsilon}$,
    for $0\le i\le n-1$;
    \item $\mathrm{Vol}(\Psi_z^j(U_z^j))
    \ge e^{-n(2m_j\chi^+_j+\ldots+2m_q\chi^+_q)}e^{-8kn\varepsilon}$ for $j\leq q$
    and  $\mathrm{Vol}(\Psi_z^{q+1}(U_z^{q+1}))
    \ge e^{-8kn\varepsilon}$.
  \end{enumerate}
\end{proposition}

Recall that a set $A$
is $(n,r_0)$-separated, if for all $x,y\in A$  we have
\[\max_{0\leq i \leq n-1} |f^{i} (x) - f^i (y)|> r_0\]
(we assume here that $|f^i (x)-f^{i}(y)|>r_0$ if at least one among $f^{i}(x)$ and $f^i (y)$ is not defined).

\begin{proof}
The statement being local, the same proof
as in \cite{D12} applies here.
\end{proof}

\begin{proof}[{Proof of Theorem \ref{theoremD}}]
We keep the notations of Proposition \ref{propositiondupond}
and set $\Psi^j_n:=\cup_{z\in\mathcal{E}_n}\Psi_z^j(U_z^j)$.
Up to taking suitable local charts,
 we can assume that  the $\Psi_z^j(U^j_z)$'s are graphs
 above $\sigma( f^{-n}(V))$ where $\sigma: \mathbb{C}^k\to \mathbb{C}^{l_j}$ is the orthogonal projection.
Let $\omega:=dd^c\|z\|^2$ be the standard K\"{a}hler
$(1,1)$-form on $V$.  For $a\in \sigma( f^{-n}(V))$, set  $$\Gamma_n(a):=\{(z,f(z),\ldots,f^{n-1}(z)),z\in \sigma^{-1}(a)
\cap f^{-n+1}(V)\}\subset V^{n}.$$
Then $\mathrm{Vol} (\Gamma_n(a))
=\int_{\Gamma_n(a)}\omega_n^{k-l_j}$,
where $\omega_n=\sum_{i=1}^{n}\Pi_i^*\omega$
and, for every $1\le i\le n$, 
$\Pi_i$ denotes
the projection of
$V^n$ onto its $i$-th factor.
Since $\mathcal{E}_n$ is  $(n, r_0)$-separated and
 $\diam f^i(\Psi_z^j(U_z^j))\le e^{-n\varepsilon}$
 the set $\Psi^j_n\cap \sigma^{-1}(a)$ is
 $(n, r_0/2)$-separated.
 Lelong inequality implies
 that $\mathrm{Vol} (\Gamma_n(a))\ge \mathrm{Card}\left(\Psi^j_n\cap \sigma^{-1}(a)\right)$,
 see for instance \cite{G03}.
By using  Proposition \ref{propositiondupond} we get

\begin{equation}\label{eq:vol-1}
\begin{aligned}
 \int_{a\in \sigma( f^{-n}(V))}
 \mathrm{Vol}  (\Gamma_n(a)) da 
 &\ge
   \int_{a\in \sigma( f^{-n}(V))}
   \mathrm{Card} \left(\Psi^j_n\cap \sigma^{-1}(a)\right) da \\
   & =  \mathrm{Vol} (\sigma(\Psi^j_n))
    \ge e^{n(h_\nu-2\varepsilon)}
   e^{-n(2m_j\chi^+_j+\ldots+2m_q\chi^+_q)}e^{-8kn\varepsilon}.
\end{aligned}
\end{equation}
On the other hand, we also have
$$\mathrm{Vol} (\Gamma_n(a))=\sum_{0\le n_i\le n-1} 
\int_{ f^{-n}(V)}
[\sigma^{-1}(a)]\wedge (f^{n_1})^*\omega
\wedge\ldots\wedge (f^{n_{k-l_j}})^*\omega.$$
By Lemma \ref{l:thesislemmaA}, 
and since the sum in the last expression contains $n^{k-l_j}$ terms, there exists a function
$C\colon \mathbb N\to \mathbb R$ independent of $a$
such that $\limsup\limits_{n\to\infty} C(n)^{1/n}=1$
and,  for all $n\in \mathbb N$,
$$\mathrm{Vol} (\Gamma_n(a)) \le  C(n)\, \delta_{k-l_j}^n.$$
Hence,  for all $n\in \mathbb N$,
we have
\begin{equation}\label{eq:vol-2}
\int_{a\in \sigma( f^{-n}(V))} \mathrm{Vol} (\Gamma_n(a))da
\le C(n)\,\delta_{k-l_j}^n\int_{a\in \sigma( f^{-n}(V))}da 
\le \alpha(n)\, \delta_{k-l_j}^n\end{equation}
where
$\alpha(n):=C(n)\int_{a\in\sigma(V)}da=C(n) \,\mathrm{Vol} (\sigma(V))$ (notice that $\sigma( f^{-n}(V))\subset \sigma(V)$).
In particular, we have 
$\limsup\limits_{n\to\infty} \alpha(n)^{1/n}=1$.

Combining the two inequalities
\eqref{eq:vol-1} and \eqref{eq:vol-2},
we deduce that for any $\varepsilon>0$ there
exists an integer $n_0$ such that for any $n>n_0$ we have
$$\log \delta_{k-l_j}+\frac1n \log \alpha(n) \ge h_\nu(f)-(2m_j\chi^+_j+\ldots+2m_q\chi^+_q)-(8k+2)\varepsilon.$$
By letting $n\to\infty$ and $\varepsilon\to 0$,
we have
\begin{align*}
h_\nu(f) 
  & \le \log \delta_{k-l_j}+2m_j\chi^+_j+\ldots+2m_q\chi^+_q \\
   & \le \log \max_{1\le i\le  k-l_j} \delta_{i}+2m_j\chi^+_j+\ldots+2m_q\chi^+_q.
\end{align*}
The proof is complete.
\end{proof}

\begin{proof}[{End of the proof of Theorem \ref{t:lyapunov-positive}}]
Let us first assume that $f$
admits $l\geq 2$ distinct Lyapunov exponents.
By using  \eqref{inequalityD} 
and Proposition \ref{prop:BDR}
we have
\begin{equation}\label{eq:chillehnu}
    \chi^+_l
\geq
\frac{h_\nu(f)- \log \max_{1\le i\le k-m_l}\delta_i}{2m_l}
\geq
\frac{h_\nu(f)- \log d^*_{k-1}}{2{m_l}},
\end{equation}
which concludes the proof in this case.

Assume now that all the Lyapunov exponents are equal to $\chi$.
In particular, we have $l=1$ and $m_l=m_1=k$.
By Ruelle inequality,
we have $h_\nu { (f)} \leq 2k \chi$,
hence the desired estimate in this case also follows.
\end{proof}

\begin{remark}\label{r:logcon}
  In the case of endomorphisms of $\mathbb P^k$ (see \cite{dTh08}
and \cite{D12})  the denominator in Theorem \ref{t:lyapunov-positive} 
\eqref{i:2k}
can be taken to be equal to $2$ instead of $2 m_l$, even when $m_l>1$. This is a consequence of the log-concavity of the sequence $\{d^*_p\}_{0\leq p\leq k}$
(which in that case is actually linear).

Let $f$ be a polynomial-like map as in Theorem \ref{t:lyapunov-positive} and assume that the sequence $\{d_p^*\}_{0\leq p \leq k}$ is
log-concave, i.e., that we have
$d_{p-1}^*\cdot d_{p+1}^*\le (d_p^*)^2$ for $1\le p\le k-1$.
Then, the denominators  $2m_{l}$ can also be replaced by $2$ in Theorem \ref{t:lyapunov-positive} \eqref{i:2k}. 
 Indeed, the log-concavity implies that
$(d_k^*)^{m-1}\cdot d^*_{k-m}\le (d_{k-1}^*)^{m}$
for all $1\le m\le k$. 
By the last inequality and the fact that $h_\nu(f)\le \log d_t=\log d_k^*$ we obtain
$$ (m-1)  h_\nu(f)\le (m-1)  \log d_k^*\le \log\frac{(d_{k-1}^*)^{m}}{d^*_{k-m}}=m \log d_{k-1}^*-\log{d^*_{k-m}} $$
and hence
\begin{equation}\label{eq:m-1entlogd}
\frac{h_\nu(f)- \log d^*_{k-m}}{2m}\geq
\frac{h_\nu(f)- \log d_{k-1}^*}{2}.
\end{equation}

Let us first assume that $f$ admits $l\ge 2$ distinct Lyapunov exponents,  and let $m_l$ be
 the multiplicity of the smallest Lyapunov exponent,
as above. Since $\max_{1\le i\le k-m_l}\delta_i\le 
d^*_{k-m_l}$
 by Proposition \ref{prop:BDR}, thanks to \eqref{eq:m-1entlogd} we have
 $$\frac{h_\nu(f)- \log \max_{1\le i\le k-m_l}\delta_i}{2m_l}
\ge\frac{h_\nu(f)- \log d^*_{k-m_l}}{2m_l}\geq
\frac{h_\nu(f)- \log d_{k-1}^*}{2}.$$
This permits to improve the second inequality in \eqref{eq:chillehnu}, and proves the assertion in this case.

 Assume now that all the Lyapunov exponents are equal to $\chi$. Then, thanks to the arguments in the last lines of
 the proof of Theorem \ref{t:lyapunov-positive}, it is enough to prove the inequality $\frac{h_\nu(f)- \log d_{k-1}^*}{2}\le \frac{h_\nu(f)}{2k}$. Since this 
is a consequence of \eqref{eq:m-1entlogd} applied with $m=k$
(recall that $d_{0}^*=1$), 
the assertion follows in this case, too. 
\end{remark}

\section{Holomorphic families of polynomial-like maps}\label{s:fplmaps}

\subsection{General definitions}
We will consider
in all this paper holomorphic
families of polynomial-like maps. These are defined
as follows, see \cite{Pha05,DS10}.

\begin{definition}
Let $M$ be a complex manifold and $\mathcal{U},\mathcal{V}$
be connected open subsets of $M\times \mathbb{C}^k$ such that $\mathcal{U}{ \subset} \mathcal{V}$.
Let $\pi_M\colon M\times \mathbb{C}^k\to M$ be the
standard projection. Assume that for every $\lambda\in M$
we have
$\emptyset\ne U_\lambda\Subset V_\lambda\Subset\mathbb{C}^k$
with $U_\lambda$ connected and $V_\lambda$
convex, where  $U_\lambda:=\mathcal{U}\cap \pi^{-1}_M(\lambda)$
and $V_\lambda:=\mathcal{V}\cap \pi^{-1}_M(\lambda)$.
Assume also that $U_\lam$ and $V_\lambda$
depend continuously on $\lam$.
 A \emph{holomorphic family of polynomial-like maps}
 is a proper
  holomorphic map
  $f:\mathcal{U}\to\mathcal{V}$ of the form
  $(\lambda,z)\mapsto (\lambda,f_\lambda(z))$.
\end{definition}

From the definition, $f$ has a
well-defined topological degree $d_t$.
We will always assume that $d_t\ge 2$. All the maps
$f:U_\lam\to V_\lam$ are polynomial-like
with the same topological
degree $d_t$. We
denote by
$\mu_\lambda, J_\lambda$, and $K_\lambda$
the equilibrium measure, the Julia set, and  the filled
Julia set of $f_\lambda$, respectively.
Since
the function
$\lam \to K_\lam$ is upper semicontinuous
for the Hausdorff topology,
when working locally near a given $\lam_0 \in M$
we
can assume without loss of generality that
$\mathcal V$  is equal to
$M \times V$
 for some open and convex subset  $V\Subset \C^k$.
We denote by $C_f$ the critical set of $f$.
Observe that the current of integration
$[C_f]$ is given by
$dd^c \log |\jac f|$, where $\jac f$ is the Jacobian of $f$.

The following result is due to Pham \cite{Pha05}
in the case of $\nu_\lam= \mu_\lam$, see also
\cite{DS10}.

\begin{proposition}\label{p:pham}
Let 
$(f_\lam)_{\lam \in M}$
be a holomorphic family of polynomial like maps
of large topological degree. Let $\mathcal R$ be a horizontal
positive closed current on $ \mathcal V= M\times V$
of bidegree $(k,k)$.
Assume that, for all $\lam \in M$, the slice measure
$\nu_\lam := \mathcal R_\lam$ is 
 an ergodic $f_\lam$-invariant   measure.
Let $L_1 (\lam)\geq \ldots \ge L_k (\lam)$ denote the
$k$ Lyapunov exponents of
$\nu_\lam$, counting multiplicities, and,
for every $1\leq \ell\leq k$
let
\[
\Sigma_{\ell} (\lam) := \sum_{j=1}^{\ell} L_j (\lam)
\]
be the sum of the $\ell$ largest exponents
of $\nu_\lam$. If there exists $\lam_0 \in M$
such that
$L_k(\lambda_0)$ is finite, then, for all $1\leq l \leq k$,
the function $\Sigma_\ell (\lam)$
is plurisubharmonic (psh)
on $M$.
\end{proposition}

Recall that a current
$\mathcal R$ in $M\times V$
is \emph{horizontal} if
$\pi_V (\Supp \mathcal R) \Subset V$
\cite{Duj04,DS06}.
For a horizontal positive closed current of bidegree
$(k,k)$, the \emph{slice} $\mathcal R_\lam$
(which, for smooth $\mathcal R$, coincides with
 the intersection
$\mathcal R \wedge \pi_M^* \delta_\lam
=\mathcal R \wedge [\pi_M^{-1} (\{\lam\})]$)
is well-defined for all $\lam\in M$ and
can be seen as a positive measure on $V$, whose
mass does not depend on $\lam$, see \cite{DS06}.

\begin{proof}[{Proof of Proposition \ref{p:pham}}]
The proof is essentially
the same as for the case $\nu_\lam= \mu_\lam$
(see \cite{Pha05, DS10})
 hence we only sketch it.

The differential $D_z f_\lam$
depends holomorphically
on $(\lam,z)$. For every $1 \leq \ell \leq k$,
it induces
the linear map
\[\bigwedge^\ell D_z f_\lam
\quad \colon\quad
\bigwedge^\ell T_z \C^k
\to \bigwedge^\ell T_{f_\lam (z)}\C^k
\]
which is defined as
\[
\bigwedge^\ell D_z f_\lam (e_1 \wedge \ldots \wedge e_\ell) := 
D_z f_\lambda (e_1) \wedge \ldots \wedge D_z f_\lambda (e_\ell).
\]
This map still
depends holomorphically on
$(\lam,z)$.
Hence, the map
$(\lam,z)\mapsto \log \|\wedge^\ell D_z f_\lam\|$ is psh.
For every $n\geq 0,$
set $\Psi_n (\lam) :=\langle \nu_\lam,  \log \|\wedge^\ell D_z f^n_\lam\|\rangle$.
By \cite[Proposition A.1]{Pha05}
 (see also \cite[Appendix A.1]{B16}), $\Psi_n$ is psh
or equal to $-\infty$ on $M$.
For all $n,m\geq 0$, $z \in U$, and $\lam \in M$
we also have
\[
\|\wedge^\ell D_z f^{n+m}\|
\leq \|
\wedge^\ell D_{f^n_\lam (z)}  f^m
\|
\, \cdot\,
\|
\wedge^\ell D_{z} f^n
\|,
\]
which implies that
$\Psi_{n+m}(\lam)\leq \Psi_{n} (\lam)+ \Psi_m (\lam)$.
Hence, the sequence $n^{-1} \Psi_n$ decreases to
$\Psi := \inf n^{-1} \Psi_n$. By Oseledets theorem, we have
$\Psi(\lam) = \Sigma_\ell (\lam)$.
We deduce that
$\Sigma_\ell$
is psh or identically $-\infty$. The latter possibility
is excluded since
the assumption on $\lam_0$
implies that $\Sigma_\ell (\lam_0) >-\infty$.
The assertion follows.
\end{proof}

\subsection{Stability notions}

We fix in this section a connected and
simply connected
complex manifold
$M$ and a holomorphic family of
polynomial-like maps
of large topological degree $f\colon \mathcal U\to \mathcal V  = M\times V$.
We first consider
the space of maps
$$\mathcal{O}(M,\mathbb{C}^k)
:=\{\gamma:M\to \mathbb{C}^k: \gamma\ \mathrm{holomorphic}\}$$
with the topology of local uniform convergence.
 $\mathcal{O}(M,\mathbb{C}^k)$
is a metric space.
We then define
the 
subspace
\begin{equation}\label{e:J}
\mathcal{J} 
:=\{\gamma\in \mathcal{O}(M,\mathbb{C}^k): \gamma(\lambda)\in J_\lambda \, \forall \lambda\in M\}
\end{equation}
and the two
 natural maps:
\begin{enumerate}
    \item $\mathcal{F}:\mathcal{J}\to\mathcal{J}$,
    which is defined by 
    $\mathcal{F}(\gamma)(\lambda)=f_\lambda(\gamma(\lambda));$
    \item $p_\lambda: \mathcal{J}\to J_\lambda$,
    which is defined by $p_\lambda(\gamma)=\gamma(\lambda)$.
\end{enumerate}
Observe that
$\mathcal{J}$
is a compact metric space, and
that
$(\mathcal J, \mathcal F)$ is a well-defined
dynamical system.

\begin{definition}
 A \emph{web} for the family $f$
is an $\mathcal F$-invariant probability measure on $\mathcal J$.

Given $\lambda_0\in M$ and
an $f_{\lambda_0}$-invariant 
 probability measure $\nu$
supported on $J_{\lambda_0}$, a
\emph{$(\lambda_0, \nu)$-web}
(or \emph{$\nu$-web} for brevity)
is a web $\mathcal M$ such that
  $(p_{\lambda_0})_*\mathcal{M}=\nu$.

If a web $\mathcal M$ satisfies
$(p_{\lambda})_*\mathcal{M}=\mu_\lambda$
(the equilibrium measure for $f_{\lambda}$)
for all $\lambda \in M$, 
$\mathcal{M}$  is an \emph{equilibrium web} for the family $f$.
\end{definition}

\begin{definition}
Given $\lambda_0\in M$ and
an $f_{\lambda_0}$-invariant 
 probability measure $\nu$
supported on $J_{\lambda_0}$,
a dynamical lamination $\mathcal{L}$
(see Definition \ref{def:lamination})
is said to be a \emph{$\nu$-lamination} if $\nu (\{\gamma(\lambda_0):\gamma\in\mathcal{L}\})=1$.
\end{definition}

 The \emph{repelling $J$-cycles} of $f_\lambda$ are
 the repelling cycles of $f_\lambda$ which belong to $J_\lambda$.
 In dimension 
  $k=1$,
  all
 repelling cycles are automatically
 repelling $J$-cycles. When $k\ge2$ there are examples
 of repelling points outside of the Julia set (see
 for instance \cite{FS01,HP94}). Recall, however, that
repelling $J$-cycles
are dense in the Julia set
\cite{BD99,DS03}.

\begin{definition}\label{def:motion-rep-all}
  We say that the repelling $J$-cycles
  of $f_\lambda$ \emph{move holomorphically} over an open subset $\Omega\subseteq M$ if for every $n\ge 1$ there exists a set of holomorphic maps
  $\rho_{j,n}\in \mathcal J$
  such that 
  $\mathcal{R}_n(\lambda)=\{\rho_{j,n}(\lambda):1\le j\le N_d(n)\}$
  for all $\lambda\in \Omega$, where $\mathcal{R}_n(\lambda):=\{
  \mbox{repelling } n-\mbox{periodic  points of }f_\lambda \mbox{ in }J_\lambda\}$ and $N_d(n)=\mathrm{Card} (\mathcal{R}_n(\lambda))$
 for all $\lam \in M$.
  \end{definition}

  The following, a priori weaker, condition
  on the repelling cycles was introduced in \cite{B19}.

  \begin{definition}\label{def:motion-rep-asympt}
  We say that \emph{asymptotically all} repelling cycles move holomorphically  on $\Omega$
  if
  there exists a set
  $\mathcal P= \cup \mathcal P_n \subset \mathcal J$
  with the following properties:
 \begin{enumerate}
 \item  $\mathrm{Card }(\mathcal P_n) = d_t^n + o(d_t^n)$;
 \item every $\gamma\in \mathcal P_n$ is $n$-periodic;
 \item  for all open $\Omega'\Subset \Omega$, we have
\[ \frac{\mathrm{Card } \{
 \gamma 
 \in \mathcal P_n \colon\gamma 
 (\lam)\mbox{ is repelling for all } \lambda \in \Omega'\} }{d_t^n}\to 1.\]
\end{enumerate}
\end{definition}

\begin{definition}
  We say that
  $\lambda_0\in M$ is a \emph{Misiurewicz} parameter if there exist
  integers $p_0,n_0\ge 1$ and a holomorphic map $\sigma$ defined
  on some neighbourhood of $\lambda_0$ such that
  $\sigma(\lambda)\in \mathcal{R}_{p_0}(\lambda)$
  and $\Gamma_\sigma\cap W\ne \emptyset$
  but $\Gamma_\sigma\not\subseteq W$ for some
  irreducible component $W$ of $f^{n_0}(C_f)$,
  where $\Gamma_\sigma$ denotes the graph of $\sigma$.
\end{definition}

The following result is a generalization
of Theorem-Definition \ref{t:BBD-intro}
to the setting of polynomial-like maps of
large topological degree. Observe in particular that
it applies to
any subfamily of $\mathcal H_d (\P^k)$,
for any $k\geq 1$ and $d\geq 2$. In particular,
it permits to extend the
definition of stability to these settings.

\begin{theorem}[{\cite[Theorem C]{B19}}]\label{equivth}
Let $M$ be a connected and simply connected complex manifold
and let $(f_\lambda)_{\lambda \in M}$ be a holomorphic family of polynomial-like maps of large topological degree.
Then the following assertions are equivalent:
\begin{enumerate}
    \item[\bf (S1)] asymptotically all repelling $J$-cycles of $f_\lambda$ move holomorphically over $M$;
    \item[\bf (S2)] $dd^c L_{\mu}(\lambda)\equiv0$, where $L_{\mu}(\lambda):=\int_{U_\lambda}\log\left|\jac f_\lambda(z)\right|\mu_\lambda(z)$ is the sum of the Lyapunov exponents of the unique measure of maximal entropy $\mu_\lambda$  of $f_\lambda$;
    \item[\bf (S3)] there are no Misiurewicz parameters;
    \item[\bf (S4)] the family admits an equilibrium lamination.
\end{enumerate}
\end{theorem}

Condition {\textbf {(S3)}}
is already present in \cite{BBD18} (and is in particular
equivalent to the conditions in Theorem-Definition \ref{t:BBD-intro}),
see also \cite{Lev82,Ly83a} for the analogous
equivalence in dimension 1 and \cite{BB22} for further characterization of stability.

 \begin{remark}
  By \cite{B18},
  condition {\textbf {(S1)}}
  can actually be weakened 
 to the following, a priori weaker, condition:
  there exists a function $N\colon \mathbb N \to \mathbb N$
  with $\limsup_{n\to \infty} d^{-n}_t N(n)>0$ such that,
  for every $n$, $N (n)$ repelling periodic
  points move holomorphically
  (as repelling periodic points). 
\end{remark}

A crucial role in the proof of both Theorems
\ref{t:BBD-intro} and \ref{equivth}
is the concept of \emph{acritical web}. While the
definition in \cite{BBD18} and \cite{B19}
is given only for equilibrium webs,
we can give it for general webs.

\begin{definition}\label{d:acritical}
  A web $\mathcal{M}$ is said to be \emph{acritical} if
  $\mathcal M (\mathcal J_s)=0$, where
  $$\mathcal J_s := \{\gamma\in\mathcal{J}:\Gamma_\gamma\cap GO(C_f)
  \ne \emptyset\} 
  \quad  \mbox{ and } 
  \quad GO ( C_f) := \cup_{n,m\geq 0} f^{-m} ( f^{n} (C_f)).$$
 \end{definition}

In particular (see for instance
\cite[Theorem 4.11]{B19}),
the conditions
{\bf (S1)}-{\bf (S4)}
in Theorem \ref{equivth} are equivalent to the following one:
\begin{itemize}
    \item[\bf (S5)] there exists an acritical equilibrium web.
   \end{itemize}

\smallskip

The following is a more general version of
Theorem 
\ref{t:main-intro} that was announced in the Introduction.
 For any $\lam\in M$, 
we denote by $C_{f_\lam}^+:=\cup_{m\ge0} f_\lam^m(C_{f_\lam})$ the postcritical set of $f$.

\begin{theorem}\label{t:mainequiv}
Let $M$ be a connected and simply connected complex manifold
and let $(f_\lambda)_{\lambda \in M}$ be a holomorphic family of polynomial-like maps of large topological degree.
 Fix $\lam_0\in M$ and consider an ergodic
$f_{\lambda_0}$-invariant
 probability 
measure $\nu_0$ such that $\Supp\nu_0\subseteq J_{\lambda_0}$, the smallest Lyapunov exponent of $\nu_0$ is strictly
positive,  and $\nu (C^+_{f_{\lambda_0}})=0$. Then, up to replacing $M$ with a sufficiently small open
neighbourhood $M_{\lambda_0, \nu_0}$
of $\lambda_0$,
the conditions {\bf (S1)}-{\bf (S4)} in
Theorem \ref{equivth}
and the condition {\bf (S5)} above
are 
equivalent to the following assertions:
\begin{itemize}
  \item[\bf (S4')] there exists a $\nu_0$-lamination;
  \item[\bf (S5')] there exists an ergodic acritical $\nu_0$-web
  $\mathcal M$
  such that, for all $\lam \in M_{\lam_0, \nu_0}$,
     the
   Lyapunov exponents 
   of $(p_\lam)_* \mathcal M$ 
   are uniformly bounded from below by a strictly positive constant.
\end{itemize}
\end{theorem}

 The main reason why we may need to reduce $M$ to $M_{\lambda_0, \nu_0}$
 (possibly depending on $\nu_0$)
is due to the fact that the smallest Lyapunov exponent of the motion of $\nu_0$ may, a priori, become negative at some $\lambda_1\in M$. In our construction (and more precisely
in the construction of the $\nu_0$-lamination)
we will need to restrict to the
parameters where such an exponent stays positive.

\smallskip

It is proved in \cite{BBD18}
that the existence
of a graph $\gamma\colon M\to \P^k$
whose orbit does not intersect the postcritical set
implies the stability of the family, and is then
equivalent to it. The same result is proved in
\cite{B16} for families of polynomial-like maps.
In particular, it follows directly
that the existence of an acritical web 
(associated to any measure, hence in particular {\bf (S4')})
implies that
the family is stable. 
The implication
{\bf (S5')} $\Rightarrow$ {\bf (S4')}
 follows from similar arguments as in \cite{BBD18,B19,BB18}, that we will briefly recall for convenience and later reference,
 see Section \ref{s:new-lamination} and the Appendix.
The main point of Theorem \ref{t:mainequiv}
is the implication
{\bf (S1)} $\Rightarrow$ {\bf (S5')}, and in particular the positivity of the Lyapunov exponents in {\bf (S5')}.

\medskip

By Theorem \ref{t:lyapunov-positive}, 
Theorem \ref{t:mainequiv} in particular applies when the measure-theoretic entropy of the measure  $\nu_0$ is strictly larger than
$\log d^{*}_{k-1}$. In this case
we can also find a uniformity for the neighbourhood
$M_{\lam_0,\nu_0}$, depending only on the measure-theoretic entropy of $\nu_0$.
As we will see,
this fact is also a consequence of the uniform bound on the Lyapunov exponents 
given by Theorem \ref{t:lyapunov-positive}.
The existence of $M_{\lam_0,h}$ as in the statement below follows from the upper semicontinuity of 
 the function $\lambda \mapsto d^*_{k-1} (f_\lam)$.

\begin{theorem}\label{t:mainlamin}
 Let $M$ be a connected and simply connected
 complex manifold
and let $(f_\lambda)_{\lambda \in M}$ be a 
 stable
family of polynomial-like maps of
 large
topological
degree. Fix $\lambda_0\in M$ and
 $h \in \mathbb R$ such that $\log d^*_{k-1} (f_{\lam_0})<h < \log d_t$ and let
 $M_{\lam_0,h}$ be a simply connected open neighbourhood of $\lam_0$ such that $\log d^*_{k-1} (f_{\lam})<h $ for any $\lam\in M_{\lam_0,h}$.
Then for
 any
 ergodic
$f_{\lambda_0}$-invariant probability
measure  $\nu_0$
supported in $J_{\lambda_0}$
such that $h_{\nu_0}(f_{\lam_0})>h$,
 the properties
{\bf (S4')} and {\bf (S5')}
hold on $M_{\lambda_0,h}$.

Moreover, 
there exists a dynamical
lamination
$\mathcal L$
satisfying
\begin{itemize}
    \item[\bf (Sh)]
    $\nu (\{\gamma(\lambda):\gamma\in\mathcal{L}\})=1$
for
every
$\lam \in M_{\lambda_0, h}$ and
every
$f_\lam$-invariant measure
$\nu$ 
such that $h_\nu (f_\lam)> h$.
\end{itemize}
\end{theorem}

The following is a version of the above result for families of endomorphisms of $\P^k$, stating that, in this case, 
the neighbourhood $M_{\lam_0, h}$ in Theorem \ref{t:mainlamin} can be taken equal to $M$
(recall that
$d^*_{k-1}(f_\lam)=d^{k-1}$ for all $\lam\in M$ in this case). 
Observe that Theorem \ref{t:main-intro}
corresponds to the case 
of Corollary \ref{c:pk} 
where $M$ is an open subset
of $\mathcal H_d (\P^k)$.

\begin{corollary}\label{c:pk}
 Let $M$ be a connected and simply connected
 complex manifold
and let $(f_\lambda)_{\lambda \in M}$ be a
stable
family of endomorphisms of $\mathbb P^k$ of algebraic degree $d\geq 2$.
Fix $\lam_0\in M$ and let $\nu_0$ be any ergodic 
$f_{\lam_0}$-invariant probability measure 
with $h_{\nu_0} (f_{\lam_0})> (k-1)\log d$. Then, the properties 
{\bf (S4')} and {\bf (S5')}
hold on $M$. In particular,
there exists a dynamical
lamination
$\mathcal L$
satisfying 
\begin{itemize}
    \item[\bf (S*)]
    $\nu (\{\gamma(\lambda):\gamma\in\mathcal{L}\})=1$
for every $\lam \in M$ and
every $f_\lam$-invariant measure
$\nu$ 
 such that $h_\nu (f_\lam)> (k-1)\log d$.
\end{itemize}
\end{corollary}

\section{Proof of main results}
\label{s:proof-thm}

In this section we give the proof of
Theorems \ref{t:mainequiv} and \ref{t:mainlamin} and Corollary \ref{c:pk}. In particular, this also proves Theorem \ref{t:main-intro}.
We fix a connected and simply connected complex manifold $M$ and let $(f_\lam)_{\lam \in M}$ be a stable family of polynomial-like maps of large topological degree.
In Section \ref{s:statoac} we prove the existence of
a special acritical $\nu$-web, see Proposition \ref{p:nuweb}.
In Section \ref{sec:Lyapunov} we define a Lyapunov
function for a given $\nu$-web and 
give a criterion for the
pluriharmonicity of such
function, that in particular applies to the web constructed in Proposition \ref{p:nuweb}. 
As a consequence, 
we deduce the positivity of the Lyapunov exponents of the slices of that web
in a neighbourhood of the starting parameter, proving 
{\bf (S5')}.
 In Section \ref{s:new-lamination},
we prove the existence of
a $\nu$-lamination and conclude the proofs of
the main results.

\subsection{From stability to acritical $\nu$-webs}
\label{s:statoac}
In this section we 
 give the first part of the proof of the implication
{\bf (S1)} $\Rightarrow ${\bf (S5')} in Theorem \ref{t:mainequiv}, namely we prove the existence of a
suitable acritical web (with no requirement on the positivity of the associated Lyapunov exponents) 
under the stability assumption.
 Observe that
 we do not need to restrict the parameter space to get such property. For simplicity, we will set the following definition.
 Recall that we fix a stable family $(f_\lam)_{\lam \in M}$ in all this section.

\begin{definition}\label{def:critical-inside}
Given $n \in \mathbb N$, a measure $\mathcal M$ on $\mathcal J$ is \emph{critically $n$-aligned} if
\[
\forall 
\gamma\in \mathrm{Supp}
\mathcal{M} 
\colon \quad 
 \Gamma_\gamma\cap 
 f^{n}(C_f)\ne \emptyset
\Rightarrow \Gamma_\gamma\subseteq f^n(C_f).
\]
$\mathcal M$ is \emph{critically aligned} if it is critically $0$-aligned. It is
\emph{postcritically aligned} if it is critically $n$-aligned for all $n\geq 0$.
\end{definition}

\begin{lemma}\label{l:pc-acritical}
Take $\lam_0\in M$ and let
$\nu$ be an $f_{\lam_0}$-invariant measure such that $\nu (C^+_{f_{\lam_0}})=0$.
Let 
 $\mathcal M$ be a postcritically aligned
 $\nu$-web. Then $\mathcal M$ is acritical.
\end{lemma}

\begin{proof}
By the $\mathcal F_*$-invariance of $\mathcal M$,
it is enough to
prove that $\mathcal M (\mathcal J^+_s)=0$, where
$\mathcal{J}_s^+
:=\left\{\gamma\in\mathcal{J}\colon \Gamma_\gamma\cap C_f^+\ne\emptyset\right\}$.
Since $\mathcal M$ is postcritically aligned, by Definition \ref{def:critical-inside}
we have
$$\mathcal{M}\left(\left\{\gamma\in\mathcal{J}
\colon \Gamma_\gamma\cap C_f^+\ne\emptyset\right\}\right)
\le
\mathcal{M}\left(\left\{\gamma\in\mathcal{J}\colon
\Gamma_\gamma\subseteq
C_f^+\right\}\right)$$
$$=\mathcal{M}\left(\left\{\gamma\in\mathcal{J}\colon \gamma(\lambda_0)\in C_{f_{\lambda_0}}^+\right\}\right)
=\nu \left(C_{f_{\lambda_0}}^+\right)=0,$$
where the last equality
follows from the assumption on $\nu$.
The assertion follows.
\end{proof}

The main result of this section is the following proposition.

\begin{proposition}\label{p:nuweb}
Fix $\lambda_0\in M$ and
let $\nu$ be an ergodic
$f_{\lambda_0}$-invariant probability
measure supported in $J_{\lambda_0}$ and
such that $\nu(C^+_{f_{\lam_0}})=0$.
There exists a 
postcritically aligned ergodic
$\nu$-web $\mathcal M$.
\end{proposition}

Observe that, in particular, 
$\mathcal M$ as in Proposition \ref{p:nuweb}
is acritical by Lemma \ref{l:pc-acritical}, and for every $\lam \in M$
the measure $\nu_\lambda:=(p_{\lambda})_*\mathcal{M}$
is an ergodic $f_\lam$-invariant
probability measure.

We will need
 the following technical lemma.
\begin{lemma}\label{limitlemma}
Let $X$ be a metric space and  $\nu$ be
a compactly supported probability measure on $X$.
Let $X_j:=\{x_1^j,x_2^j,\ldots,x^j_{l(j)}\}$
be a sequence of finite sets
such that
\begin{enumerate}
\item the  cardinalities $l(j)$ of $X_j$
satisfy
$l(j+1)\ge l(j)$;
\item $L:=\overline{\cup_j X_{j}}$ is a
compact subset of $X$ and $\Supp \nu \subseteq L$;
\item\label{(3)}
 $\overline{\cup_n X_{j_n}}=L$ for any
 subsequence $\{X_{j_n}\}$.
\end{enumerate}
Then there exists a sequence of sets
$A_j:=\{a^j_1,a^j_2,\ldots,a^j_{l(j)}\}$
of non-negative real numbers
with
$\sum_{m=1}^{l(j)} a^j_m=1$ such that
$$\nu_j:=\sum_{m=1}^{l(j)} a^j_{m}\delta_{x^j_m}\to \nu,$$
where $\delta_{x}$ is the Dirac mass at $x\in X$.
\end{lemma}

\begin{proof}
For every $j$, let
 $\zeta_j:=\{B^j_1,B^j_2,\ldots,B^j_{l(j)}\}$
be a measurable partition of $L$ such that $x^j_m\in B^j_m$
for all $m$.
By the assumption (\ref{(3)}) it follows
the union of the $X_j$'s is dense in $L$. Hence,
we can assume that the maximum diameter of
the elements of the partition $\zeta_j$,
 goes to $0$ as $j\to \infty$,
 i.e., that \begin{equation}\label{eq-diam-partition-tend-0}
\diam \zeta_j := \sup_{1\leq m
\leq l(j)} \diam {B^j_m}\to 0
\mbox{ as } j\to \infty.
\end{equation}
For all $j$ and
$1\le m\le l(j)$,
set $a^j_m:=\nu(B^j_m)$.
By 
construction we have $|\nu_j|=1$ for all $j$. Hence,
by Banach-Alaoglu theorem,
there exists a converging subsequence
$\{\nu_{j_i}\}$ of $\{\nu_{j}\}$.
Fix one such subsequence and denote
$\widetilde \nu := \lim_{i\to \infty} \nu_{j_i}$.
Since $|\nu_{j_i}|=1$
it follows that  $|\widetilde \nu|= 1$.
Hence,
it is enough to prove that
$\nu \leq \widetilde \nu$. In order to do this, it is enough
to show that $\nu (D) \leq \widetilde \nu (D)$ for all
closed balls $D \subseteq X$
centered on $\Supp \nu$.

Let us fix a closed ball $D$ as above and, for all
$\varepsilon>0$, let $D_\varepsilon$
be the closed $\varepsilon$-neighbourhood of $D$, i.e.,
$$D_\varepsilon:=\{x\in X:\mathrm{dist}(x,D)
\le \varepsilon\}.$$
By \eqref{eq-diam-partition-tend-0},
there exists $i_0$ such that
$\diam \zeta_{j_i} < \varepsilon /2$
for all $i>i_0$. It follows that, for $i\geq i_0$,
all elements of $\zeta_{j_i}$ intersecting
$D$ are contained
in $D_\epsilon$. This implies that,
for all $i\geq i_0$, we have
$\nu(D)\le\nu_{j_i}(D_{\varepsilon})$.
Hence, $\nu(D)\le\widetilde \nu (D_{\varepsilon})$.
Finally, since
$\lim\limits_{\varepsilon\to 0}\widetilde{\nu}(D_\varepsilon)
=\widetilde{\nu}(D)$,
we obtain the desired inequality
$\nu(D)\le\widetilde \nu (D)$
by letting $\varepsilon$
tend to $0$. The proof is complete.
\end{proof}

\begin{proof}[Proof of Proposition \ref{p:nuweb}]
Since $(f_\lam)_{\lam\in M}$ is stable, by Theorem \ref{equivth},
for every $n\in \mathbb{N}^*$ there exists a collection
$\{\gamma_{j,n}:1\le j\le N_d(n)\}\subset \mathcal{J}$
such that
$$\mathcal{R}_n(\lambda)
=\{\gamma_{j,n}(\lambda):1\le j\le N_d(n)\}$$
for all $\lambda\in M$, where
$\mathcal R_n (\lambda)$
is a set of repelling $n$-periodic points
in $J_{\lambda}$ and
$N_d(n)\sim d_t^n$.
Let $A_{n}=\{a^n_1,\ldots,a^n_{N_d(n)}\}$ be a sequence of sets
of real numbers given by applying
Lemma \ref{limitlemma} with
the measure $\nu$ and the sequence of sets
$X_{n}=\mathcal{R}_n(\lambda_0)$.
Observe that the
assumptions
in Lemma \ref{limitlemma}
are satisfied by the asymptotics
of $N_d(n)$
and the equidistribution of periodic points
with respect to $\mu_{\lam_0}$
on $J_{\lam_0}\supseteq \Supp \nu$, see \cite{BD99, DS03, DS10}.
Set
$$\mathcal{M}_{n}
:=\sum\limits_{j=1}^{N_d(n)}a_j^n\delta_{\gamma_{j,n}}.$$
By definition, each
$\mathcal{M}_n$ is a discrete
probability measure supported on
$\mathcal{J}$.
In particular, $\cup_n \mathrm{Supp}\mathcal{M}_n$
is relatively compact. Hence,
by Banach-Alaoglu theorem, there exists a
converging subsequence $\mathcal{M}_{n_l}\to \widetilde{\mathcal{M}}$.
Observe that also
$\widetilde{\mathcal{M}}$ is a
measure on $\mathcal{J}$.

We claim that
 $\widetilde {\mathcal M}$ is postcritically aligned.
Indeed, by the assumption
on the motion of the repelling cycles
and Theorem
\ref{equivth},
there are no Misiurewicz parameters in $M$. So,
 $\mathcal{M}_{n_l}$ is postcritically aligned
for all $l$.
It then follows from Hurwitz theorem
that 
also 
$\widetilde{\mathcal{M}}$ is postcritically aligned.

By construction, we have 
$$(p_{\lambda_0})_*\widetilde{\mathcal{M}}
=\lim_{l\to\infty}(p_{\lambda_0})_*\mathcal{M}_{n_l}=\nu.$$
On the other hand, 
a priori, the measure $\widetilde {\mathcal M}$ may not be $\mathcal F$-invariant (hence, it may not be a web). We define $\mathcal M$ to be any limit of a subsequence of 
$n^{-1} \sum_{j=0}^{n-1}\mathcal{F}^j_*\widetilde{\mathcal{M}}$.
Then, we have $\mathcal F_*\mathcal M=\mathcal M$. 
 Since $\nu$ is $f_{\lam_0}$-invariant, for every $j\in \mathbb N$
 we have
 \[(p_{\lambda_0})_*\mathcal{F}^j_*\widetilde{\mathcal{M}}
 =
 (f^j_{\lam_0})_* \big( (p_{\lam_0})_* \widetilde{\mathcal M}
\big) =
(f^j_{\lam_0})_*\nu
=\nu.\] 
In particular, for every $n\in \mathbb N$ we have
  $(p_{\lambda_0})_*\left(n^{-1}\sum_{j=0}^{n-1}\mathcal{F}^j_*\widetilde{\mathcal{M}}\right)=\nu.$ It follows that $\mathcal{M}$ is a $\nu$-web. 
  As above, since ${\mathcal M}$ is a limit of  postcritically aligned 
  measures, it is 
   postcritically aligned, too. 
  Up to replacing ${\mathcal M}$ with one 
  of its ergodic components (by means of Choquet theorem), we can also assume that ${\mathcal M}$ is a 
   postcritically aligned ergodic web.
   The assertion follows.
\end{proof}

\begin{remark}
One can also construct webs as in Proposition \ref{p:nuweb} by using the
equidistribution of preimages
of generic points instead of repelling points, see 
\cite{DS03}.
Indeed, as mentioned above, by \cite{BBD18,B16}
the stability of a family implies
the existence of 
an element $\sigma \in \mathcal J\setminus \mathcal J_s$, i.e.,
a holomorphic map $\sigma\colon M\to V$
whose graph does not intersect the postcritical set.
Hence, one
can consider the measures
$$ \mathcal{M}_{\gamma, n}
:=d^{-n}_t
\sum_{\sigma\in \mathcal{F}^{-n}(\gamma)}a^n_\sigma\delta_\sigma$$
where $\{a^n_\sigma\colon n\in \mathbb N, \sigma\in \mathcal{F}^{-n}(\gamma)\}$
is chosen so that $(p_{\lambda_0})_*\mathcal{M}_{\gamma,n} \to \nu$.
Observe that the support of $\mathcal M_n$ is discrete and disjoint from $\mathcal J_s$.
In particular, every $\mathcal M_{\gamma, n}$ is postcritically aligned. Hence,
as in the proof of Proposition \ref{p:nuweb}, we can use the sequence $(\mathcal M_{\gamma,n})_{n \in \mathbb N}$ to construct a web $ \mathcal{M}$ satisfying the properties in Proposition \ref{p:nuweb}
\end{remark}

The following is
an immediate consequence
of Proposition \ref{p:nuweb}
and Corollary \ref{corol-mass-0}. Observe that, also in this case, $\mathcal M$ as in the statement
is acritical by Lemma \ref{l:pc-acritical}.

\begin{corollary}
Fix $\lambda_0\in M$ and
let $\nu$ be an ergodic
$f_{\lambda_0}$-invariant probability
measure supported in $J_{\lambda_0}$ such
that $h_\nu(f_{\lam_0})>\log d^*_{k-1}(f_{\lambda_0})$.
There exists
a
 postcritically aligned
 ergodic
$\nu$-web $\mathcal{M}$.
\end{corollary}

\subsection{Positivity of Lyapunov exponents
}\label{sec:Lyapunov}

Once the existence of 
 a web as in Proposition \ref{p:nuweb} is established, in order to apply the ideas in \cite{BBD18} to
prove {\bf (S5')}
we need to check that
the function associating to every $\lam$ the smallest Lyapunov exponent
of $(p_\lam)_* \mathcal M$
is 
locally uniformly bounded from below
by some strictly positive constant 
in a neighbourhood of the starting parameter
$\lam_0$.
 In order to do this,
it is enough to show that this function is 
 strictly positive and
lower semicontinuous in a neighbourhood of $\lambda_0$ (possibly depending on $\nu_0$).

\medskip

Given a web $\mathcal M$,
consider the current on $\mathcal V$ given by
\begin{equation}\label{eq:WM}
W_{\mathcal{M}}:=\int[\Gamma_\gamma]d\mathcal{M}(\gamma),
\end{equation}
where $[\Gamma_\gamma]$ is the current of
integration on the graph of the map $\gamma$.
Observe that $ W_{\mathcal M}$  is positive and closed.

 \begin{definition}\label{d:lyapunov}
Let $\mathcal M$ be a web.
We define the function $L_{\mathcal M}\colon M\to \mathbb R$
as
 \begin{equation*}
 L_{\mathcal M} 
 = \pi_* (\log |\jac f_{\lambda}(z)| W_{\mathcal M}),
 \end{equation*}
where
$W_{\mathcal M}$ is as in \eqref{eq:WM}
and $\pi$
denotes as usual the projection on the parameter space $M$.
\end{definition}

Up to restricting to a small
neighbourhood of a given parameter $\lam_0\in M$,
 we can assume that 
 $W_{\mathcal M}$
 is horizontal. Hence, the function $L_{\mathcal M}$ is well-defined.
Moreover,
 since $\log |  \jac f_{\lambda}(z)|$
is a psh function on $\mathcal V$,
by \cite[Theorem 2.1]{DS06} and
\cite[Proposition A.1]{Pha05}, the
function $L_{\mathcal M}$
is psh on $M$, or identically equal to $-\infty$.
Observe that, when $(p_\lam)_{*}\mathcal M$ is ergodic for all $\lam \in M$, the function $L_{\mathcal M}(\lam)$
is equal to the sum of the Lyapunov exponents of $(p_\lam)_{*}\mathcal M$.

 \begin{lemma}\label{ddcl=0}
Take $\lam_0 \in M$ and assume
that $\nu$ is an ergodic $f_{\lam_0}$-invariant measure
supported in $J_{\lam_0}$ and such that
the function $\log |\jac f_{\lam_0}|$
is $\nu$-integrable.
If $\mathcal{M}$ is a
critically aligned
$\nu$-web,
then $dd^c L_{\mathcal{M}}(\lambda)=0$.
 \end{lemma}

 \begin{proof}
Since the problem is local, 
we can fix $\lam_0 \in M$ and work on
  a small ball $B$ around $\lambda_0\in M$
  in the parameter space.
  The current as in \eqref{eq:WM} is then horizontal.
  We can also assume
  that $B$ has dimension 1.
 Define $J(\lambda,z):=\jac f_{\lambda}(z)$. Since the function 
 $\log |J(\lambda_0,\cdot)|$
 is $\nu$-integrable,
  by
\cite[Theorem A.2]{Pha05},
 the currents
 $\log |J(\lambda,z)| W_{\mathcal{M}}$
 and
  $dd^c (\log |J(\lambda,z)| W_{\mathcal{M}})$
 are well-defined.
 Since
 $L_{\mathcal M} = \pi_*(\log |J(\lambda,z)| W_{\mathcal{M}})$,
  to get the assertion
  it is
  enough to show that
  $dd^c(\log |J(\lambda,z)| W_{\mathcal{M}})=0$ on $B\times V$. 
 Observe that 
 $dd^c(\log |J(\lambda,z)| W_{\mathcal{M}})$ is a positive measure on $B\times V$.
  
 For this, we can follow the strategy in
\cite[Proposition 3.5]{BBD18}. For  any
small
$\varepsilon>0$,
set
$$S_\varepsilon
:=\{\gamma\in \mathrm{Supp} \mathcal{M}
:\Gamma_\gamma\cap\{|J(\lambda,z)|
<\varepsilon\}\ne\emptyset\}$$ and
define
\[W_{\mathcal{M}}^\varepsilon
:=\int_{S_\varepsilon}[\Gamma_\gamma]
d\mathcal{M}(\gamma)
\quad
\mbox {and } \quad \widetilde{W}_{\mathcal{M}}
:= W_{\mathcal{M}}-W_{\mathcal{M}}^\varepsilon.\]
For any smooth
  test function $\phi$ compactly supported in $B\times V$,
  we have
\begin{equation}\label{eq:decJW}
\begin{aligned}
 \langle dd^c (\log |J(\lambda,z)| W_{\mathcal{M}}),
 \phi \rangle
 &=
   \langle\log |J(\lambda,z)| W_{\mathcal{M}},
   dd^c \phi \rangle \\
   &=   \langle\log |J(\lambda,z)|W_{\mathcal{M}}^\varepsilon,
   dd^c \phi \rangle
   + \langle\log |J(\lambda,z)| \widetilde{W}_{\mathcal{M}},
   dd^c \phi \rangle \\
   &=:  I_1(\varepsilon)+I_2(\varepsilon).
   \end{aligned}
\end{equation}

  Since the function $\lam \mapsto \log |J(\lambda,\gamma(\lambda))|$ is 
  harmonic
  for any
  $\gamma\in{\mathrm{Supp} \mathcal{M}\setminus S_\varepsilon}$,
  we have
  $\int_B\log |J(\lambda,\gamma(\lambda))|dd^c(\phi\circ \gamma)=0$ for every such $\gamma$.
  This gives
\[
  I_2(\varepsilon)
  = \langle\log |J(\lambda,z)|\widetilde{W}_{\mathcal{M}}, 
  dd^c \phi \rangle
  =  \int_{\mathrm{Supp} \mathcal{M}\setminus S_\varepsilon}
  d\mathcal{M}(\gamma)
  \left(\int_B\log |J(\lambda,\gamma(\lambda))|
  dd^c(\phi\circ \gamma) \right)
  = 0,
\]
i.e., the second term in the right
hand side of \eqref{eq:decJW} vanishes for all $\varepsilon>0$.

In order to conclude, we need to prove that the first term
$I_1(\varepsilon)$
tends to $0$
as $\varepsilon\to 0$.
The argument in \cite{BBD18} uses the fact that
the measure of maximal entropy $\mu_{\lam_0}$ is
the Monge-Ampère of a
$(1,1)$-current with H\"older-continuous local potentials,
and in particular
 that
it gives mass
$\lesssim \varepsilon^a$ (for some positive $a$)
to balls of radius $\varepsilon$,
but this is actually
not necessary. The following claim 
will
be enough
to complete the proof.

\medskip

\textbf{Claim 1.} For any $\varepsilon\ll 1$
there exists  $c(\varepsilon)>0$
with $c(\varepsilon)\to0$
as $\varepsilon\to 0$ such that 
\begin{equation}\label{eq:moderete}
    \nu\left(\left(C_{f_{\lam_0}}\right)_\varepsilon\right)
    \le  \frac{c(\varepsilon)}{|\log|\varepsilon||}
\end{equation}
where $\left(C_{f_{\lam_0}}\right)_\varepsilon$
is the $\varepsilon$-neighbourhood of $C_{f_{\lam_0}}$.

\begin{proof}
Since $J(\lam_0,z)$ is holomorphic there
exists $A_1>0$ and $\varepsilon$ small enough
such that for every $z$ such that $\mathrm{dist}(z,C_{f_{\lam_0}})<\varepsilon$
we have
\begin{equation}\label{eq:Jlipshits}
    |J(\lam_0,z)|\le A_1 \mathrm{dist}(z,C_{f_{\lam_0}}).
\end{equation}
Since by assumption 
 the function 
 $ z\mapsto\log |J(\lambda_0,z)|$ is $\nu$-integrable,
\eqref{eq:Jlipshits}
 implies
that the function $z\mapsto \log \mathrm{dist}(z,C_{f_{\lam_0}})$
is also $\nu$-integrable.
Hence, the measure $\widetilde{\nu}
:=|\log \mathrm{dist}(z,C_{f_{\lam_0}})|\nu$
is finite and
satisfies
$\widetilde{\nu}\ll\nu$, i.e.,  we have $\widetilde{\nu}(E)=0$
if $\nu(E)=0$. In particular,
we have $\widetilde{\nu}(C_{f_{\lam_0}})=0$.
Thus, there exists  $c(\varepsilon)>0$ with
$c(\varepsilon)\to0$ as $\varepsilon\to 0$
such that
$\widetilde{\nu}
\left(\left(C_{f_{\lam_0}}\right)_\varepsilon\right)
\le c(\varepsilon)$. 
Hence, for all $\varepsilon$
sufficiently small, we have
$$|\log|\varepsilon||{\nu}
\left(\left(C_{f_{\lam_0}}\right)_\varepsilon\right)
\lesssim 
\widetilde{\nu}
\left(\left(C_{f_{\lam_0}}\right)_\varepsilon\right)
\le c(\varepsilon).$$
This gives \eqref{eq:moderete} and proves the claim.
\end{proof}

Let now $\left(C_{f}\right)_\varepsilon$
be the $\varepsilon$-neighbourhood of
$C_{f}$ in $B\times V$.
Since $J(\lambda,z)$ is holomorphic we can
find a constant $c_1$ such that
$\left(C_{f}\right)_\varepsilon
\subset \{|J(\lambda,z)|<c_1\varepsilon\}$.
By using {\L}ojasiewicz inequality we see that
there are constants $c_2,\beta_1>0$ such that 
$\{|J(\lambda,z)|<c_1\varepsilon\}\subset
\left(C_{f}\right)_{c_2\varepsilon^{\beta_1}}$.
Again by using {\L}ojasiewicz inequality we have 
 \[
\{
z \colon (\lam_0,z) \in
\left(C_{f}\right)_\varepsilon
\}
\subset 
\left(C_{f_{\lam_0}}\right)_{c_3\varepsilon^{\beta_2}}
\mbox {for some constants }  c_3,\beta_2>0.
\]

Fix now a ball $B'\Subset B$ centered at $\lambda_0$. 

\medskip

\textbf{Claim 2.} (see the Claim in \cite[Lemma 3.6]{BBD18})
There exists $0<\alpha\le1$ such that
$\mathrm{sup}_{B'}|\psi|\le|\psi(\lam_0)|^\alpha$
for every  holomorphic function $\psi:B\to\mathbb{D}^*$.

\medskip

Take $\gamma\in\Supp \mathcal{M}$ such that
$\Gamma_\gamma\cap C_f=\emptyset$
but $\Gamma_\gamma\cap (C_f)_\varepsilon\ne\emptyset$.
Then by using Claim 2 
for the holomorphic function $J(\lambda,\gamma(\lambda))$
we have
$\Gamma_{\gamma|_{B'}}
\subset \left(C_{f}\right)_{c_4\varepsilon^{\alpha\beta_3}}$
for some constants  $c_4,\beta_4>0$.
Since
$\mathcal M$ is critically aligned, we have
\begin{align*}
    \mathcal{M}
    \{\gamma\in\mathcal{J}
    :
    \Gamma_{\gamma|_{B'}}\cap
    \left(C_{f}\right)_\varepsilon
    \ne \emptyset\}
    &\le 
    \mathcal{M}
    \{\gamma\in\mathcal{J}:\Gamma_{\gamma|_{B'}}
    \subseteq
\left(C_{f}\right)_{c_4\varepsilon^{\alpha\beta_3}}\}\\
    & \le \mathcal{M}
    \{\gamma\in\mathcal{J}:(\lambda_0,\gamma(\lambda_0))
    \in \left(C_{f}\right)_{c_4\varepsilon^{\alpha\beta_3}}\}\\
    &\le  
    \nu
    \left\{\left(C_{f_{\lam_0}}\right)_{c_3(c_4\varepsilon^{\alpha\beta_3})^{\beta_2}}\right\}.
\end{align*}
Setting $\varepsilon' :=c_3(c_4\varepsilon^{\alpha\beta_3})^{\beta_2}$
for simplicity, we can now apply Claim 1 to get
\[
\mathcal{M}
\{\gamma\in\mathcal{J}:\Gamma_{\gamma|_{B'}}
\cap \left(C_{f}\right)_\varepsilon\ne \emptyset\}
\le
\nu
\left\{\left(C_{f_{\lam_0}}\right)_{c_3(c_4\varepsilon^{\alpha\beta_3})^{\beta_2}}\right\}
\le \frac{c(\varepsilon')}{|\log \varepsilon'|}
\lesssim \frac{c(\varepsilon')}{\log |\varepsilon|}
=\frac{c'(\varepsilon)}{\log |\varepsilon|}
\]
for some positive function $c'(\varepsilon)$
with $c'(\varepsilon)\to 0$ as $\varepsilon \to 0$.

Finally, by using the last inequality and 
the fact that
\[S_\varepsilon\subset S_{c_5\varepsilon^{\beta_4}}':=\{\gamma\in\mathcal{J}:
\Gamma_{\gamma|_{B'}}\cap \left(C_{f}\right)_{c_5\varepsilon^{\beta_4}}\ne \emptyset\} \mbox{ for some positive constants } c_5,\beta_4\]
(which again follows from 
{\L}ojasiewicz inequality),
setting 
$\varepsilon'' = c_5 \varepsilon^{\beta_4}$ for simplicity
we deduce that
\begin{align*}
    I_1(\varepsilon) = \langle\log |J(\lambda,z)|W_{\mathcal{M}}^\varepsilon, dd^c \phi \rangle & \lesssim |\log|\varepsilon|| \int_{ S_\varepsilon} d\mathcal{M}(\gamma)=|\log|\varepsilon|| \cdot W_{\mathcal{M}}^\varepsilon(S_\varepsilon)\\
    &\lesssim |\log|\varepsilon|| \cdot W_{\mathcal{M}}^\varepsilon(S'_{\varepsilon''})
    \le |\log|\varepsilon||  \cdot
    \frac{ c'(\varepsilon'')}{|\log|\varepsilon''||}
    \lesssim c'(\varepsilon'').
\end{align*}
Since $c'(\varepsilon'')\to 0$ as $\varepsilon\to 0$,
the assertion follows.
\end{proof}

\begin{corollary}\label{corol-lowersc}
 Let 
 $\mathcal M$ be a critically aligned web. Assume that
 $(p_\lam)_* \mathcal M$ is ergodic for all $\lam \in M$
 and denote by 
 $L^1_{\mathcal{M}}(\lambda)\ge\ldots\ge L^k_{\mathcal{M}}(\lambda)$
the
$k$
Lyapunov exponents
of
$(p_\lambda)_*\mathcal{M}$, counting multiplicities.
Assume also that there exists $\lam_0\in M$
such that
$L_{\mathcal M}^k(\lam_0)>-\infty$.
 Then the
 function $\lambda \mapsto -L^k_{\mathcal{M}}(\lambda)$
 of $\mathcal{M}$ is  plurisubhamonic. In particular, it is upper semicontinuous.
\end{corollary}

\begin{proof}
By Proposition \ref{p:pham} the upper 
partial sums of Lyapunov exponents are
plurisubharmonic
In particular,
the function $\lam \mapsto \sum_{j=1}^{k-1}L^j_{\mathcal{M}}(\lambda)$
is plurisubharmonic. 
 Set $\nu_0 := (p_{\lam_0})_* \mathcal M$.
Since, by assumption, the smallest Lyapunov exponent of $\nu_0$ is finite, the function $\log |\jac f_{\lam_0}|$
is $\nu_0$-integrable. Thanks to Lemma \ref{ddcl=0}
we have that $L_{\mathcal{M}}(\lambda)$
is pluriharmonic. Since $-L^k_{\mathcal{M}}(\lambda)
=\sum_{j=1}^{k-1}L^j_{\mathcal{M}}(\lambda)
-L_{\mathcal{M}}(\lambda)$, it follows that the function 
$-L^k_{\mathcal{M}}(\lambda)$ is plurisubharmonic.
\end{proof}

\subsection{Existence of the lamination and proofs of the main results}\label{s:new-lamination}
The following proposition
has the same proof as
\cite[Theorem 4.1]{BBD18}, see also
\cite[Theorem 3.4.1]{B16} and
\cite[Section 7]{BB18} for an overview of the arguments. 
We will give in the Appendix an intermediate statement, for later reference.

\begin{proposition}\label{t:lamination}
Fix $\lambda_0\in M$. Let $\nu_0$ be an ergodic
$f_{\lambda_0}$-invariant measure with $\mathrm{Supp} \nu_0\subseteq J_{\lam_0}$.
 Assume that there exists an acritical ergodic
 $\nu_0$-web $\mathcal M$
such that, for all $\lam\in M$, 
the Lyapunov exponents $L^j_{\mathcal M}(\lambda)$ of
the measures
$(p_\lam)_* \mathcal M$
 are uniformly strictly positive.
Then
there
exists a $\nu_0$-lamination on $M$.
\end{proposition}

\begin{proof}[Sketch of proof]
Thanks to the locally
uniform lower bound for the smallest Lyapunov exponents (and in particular to Proposition \ref{p:bbd-gen} below),
one can prove the following property,
see the Fact in \cite[Section 4.3]{BB18}:

\medskip

\textbf{Fact.} 
$\mathcal M
( \mathcal K_{\cap})=0$
 for every  compact subset
$\mathcal K \subseteq \mathcal J$,
where
\[\mathcal K_{\cap}
:=
\{
\gamma \in \mathcal J
\colon
\exists j \in \mathbb N, 
\exists \gamma' \in \mathcal K 
\mbox{ such that }
\Gamma_{\mathcal F^j (\gamma)}\cap \Gamma_{\gamma'}\neq \emptyset
\mbox{ and } 
\Gamma_{\mathcal F^j(\gamma)}
\neq
\Gamma_{\gamma'}
\}.\]

\medskip

Once this Fact is establish,
in order to construct 
 a lamination one can follow the arguments in \cite[Section 4.3]{BBD18} or \cite[Section 3.4.3]{B16} to
 show
 that the set
 $\mathcal L 
 := \mathcal J\setminus (\mathcal J_{\cap}\cup \mathcal J_s)$
 gives the desired lamination
 associated to $\nu$.
\end{proof}

 We can now complete the proofs of our main results.

\begin{proof}[Proof of Theorem \ref{t:mainequiv}]
 Proposition \ref{p:nuweb}
and Corollary \ref{corol-lowersc}
show that
(up to possibly restricting $M$ as in the statement of the theorem)
{\textbf {(S1)}} implies {\textbf {(S5')}}. We used here the fact that, by assumption, the Lyapunov exponents of $\nu_0$ are strictly positive, and that the function $\lambda\mapsto L^k_{\mathcal M} (\lam)$ is lower semicontinuous by Corollary \ref{corol-lowersc}.
Proposition \ref{t:lamination} shows that
{\textbf {(S5')}} implies {\textbf {(S4')}}.
As remarked after the statement
of the theorem, these conditions imply those
of Theorem \ref{equivth}. The proof is complete.
\end{proof}

 In order to prove Theorem \ref{t:mainlamin} (and Corollary \ref{c:pk}) we will need the following further lemma.

\begin{lemma}\label{l:entropy-lamination}
Take $\lam_0\in M$ and let $\nu$ be an $f_{\lambda_0}$-invariant measure. Let $\mathcal M$ be a $\nu$-web and $\mathcal L$ a $\nu$-lamination. Then,
for every $\lambda \in M$, the measure-theoretic entropy of $(p_\lam)_* \mathcal M$ is equal to $h_\nu (f_{\lam_0})$.
\end{lemma}

\begin{proof}
We denote $\nu_\lam := (p_\lam)_* \mathcal M$ for simplicity. In particular, we have $\nu=\nu_{\lam_0}$.
Let $A\subset J_{\lam_0}$ be a measurable set with $\nu_{\lam_0}(A)>0$. Set 
$$\mathcal{A}_A:=\{\gamma\in \mathcal{L}: \gamma(\lam_0)\in A\}.$$
For every $\lam \in M$,
set also $A_\lam:=p_\lam(\mathcal{A}_A)$.
Note that we have $\mathcal{M}(\mathcal{A}_A)
=\nu_{\lam_0}(A)
=\nu_\lam
(A_\lam)$ 
for all $\lam \in M$.

Fix $M\ni \lam_1\neq \lam_0$ and
let
$\xi:=\{A^i\}$ 
be a measurable partition for $\nu_{\lam_0}$. 
Define the measurable partition $(\xi)_{\lam_1}:= \{A^i_{\lam_1}\}$ for  $\nu_{\lam_1}$. 
By construction, the entropy of $\xi$ with respect 
 to
$\nu_{\lam_0}$ 
 \cite{CFS,KH97}
is equal  to the entropy of $(\xi)_{\lam_1}$ with respect
 to
$\nu_{\lam_1}$. By the definition of $\mathcal L$,
 for every $n\in \mathbb N$ we have
\[  \Big( \bigvee_{j=0}^n f_{\lam_0}^{-j}\xi\Big)_{\lam_1}=
\bigvee_{j=0}^{n}  f_{\lam_1}^{-j} (\xi)_{\lam_1},\]
where
 we recall that, for every $\lam\in M$ and a given partition $\eta$, the 
partition $\bigvee_{j=0}^n f_\lam^{-j}\eta$ is defined as
\[\bigvee_{j=0}^n f_\lam^{-j}\eta:=
\{f_\lam^{-n}(B^n)\cap \ldots \cap f_\lam^{-1}(B^1)\cap B^0\colon B^0, \ldots, B^n \in \eta\}.\]
By
the definition of measure-theoretic entropy
\cite{CFS,KH97}
we conclude that
$h_{\nu_{\lam_0}}(f_{\lambda_0})
\le h_{\nu_{\lambda_1}}(f_{\lambda_1})$.
By reversing the roles of $\lam_0$ and $\lam_1$,
we also see
that $h_{\nu_{\lambda_1}}(f_{\lambda_1})
\le h_{\nu_{\lam_0}}(f_{\lambda_0})$. So,
we have $h_{\nu_\lambda}(f_{\lambda})
= h_{\nu_{\lam_0}}(f_{\lambda_0})$
for all $\lam \in M$. The assertion follows.
\end{proof}

\begin{proof}[Proof of Theorem \ref{t:mainlamin}]
Theorems  \ref{t:mainequiv} and 
\ref{t:lyapunov-positive}
show that
{\textbf {(S1)}} implies {\textbf {(S5')}}
on a neighbourhood of $\lam_0$, a priori depending on $\nu_0$.
As in the proof of Theorem \ref{t:mainequiv}, 
Proposition \ref{t:lamination} shows that
\textbf {(S5')} implies \textbf {(S4')}, hence
\textbf {(S4')} holds on the same neighbourhood.

Since $d^{*}_{k-1} (f_\lam)$
depends upper semicontinuously on $\lambda$,
it follows that the set
\[M'_{\lam_0, h}:= \{ \lam \colon d^*_{k-1} (f_\lam)< h \}\]
is open (and non-empty by the assumption on $\lam_0$). Define $M^0_{\lam_0,h}$ to be the connected component of $M'_{\lam_0, h}$ containing $\lam_0$. It is enough to show that 
\textbf {(S4')} and 
\textbf {(S5')} hold on $M^0_{\lam_0, h}$.

Assume that this is not the case. Let $\Omega$ be the maximal open subset of $M^0_{\lam_0, h}$ where \textbf {(S4')} and 
\textbf {(S5')} hold, and let $\mathcal M$ be as in \textbf {(S5')}.
By Lemma \ref{l:entropy-lamination}, 
the measure-theoretic entropy of $(p_\lam)_* \mathcal M$
is constant on $\Omega$.
It follows
from Theorem \ref{t:lyapunov-positive}
that the 
function $u(\lam):= -L^k_{\mathcal M} (\lam)$
is uniformly bounded from above on $\Omega$ by a strictly negative constant. 
Recall that this function is plurisubharmonic on $M$ by Corollary \ref{corol-lowersc}.
We can assume that $u\geq 0$ on the boundary of $\Omega$, since otherwise 
(thanks to the upper semicontinuity
of this function), we would have $u<0$
in a neighbourhood $\Omega_0$ of a point of the boundary of $\Omega$, 
and
we could extend the lamination to an open set $\Omega'=\Omega \cup \Omega_0$ which is larger than $\Omega$.

The function $u$ is then bounded above by a strictly negative constant on $\Omega$, and satisfies $u\geq 0$ on its boundary. Up to adding a negative constant, we can assume that there exists $\lam_1$ in the boundary of $\Omega$ with $u(\lam_1)=0$.
By upper semicontinuity,
for every $\epsilon>0$, there exists a neighbourhood of $\lam_1$ 
where $u\leq \epsilon$. 
Since $\lam_1$ is a point of strictly positive density for $\Omega$, this gives a contradiction with the mean inequality for $u$ at $\lam_1$. 
Hence, the only possibility is that $\Omega$ is equal to $M^0_{\lam_0, h}$. This proves the claim.

To conclude,
we just need to show that 
{\textbf {(Sh)}} holds.
In order to do this, 
we apply the 
Fact in the proof of Proposition
\ref{t:lamination}
with $\mathcal K = \mathcal J$.
For all $\lam \in M^0_{\lambda_0, h}$, $\nu$
as in the statement, and $\mathcal M$
an acritical
$\nu$-web,
we have $\mathcal M (\mathcal J_{\cap})=0$.
Since we also have $\mathcal M (\mathcal J_s)=0$
for all  such $\mathcal M$,
we have $\mathcal M (\mathcal J_\cap \cup \mathcal J_s)=0$
for all $\lam,\nu, \mathcal M$ as above.
It follows that 
$\mathcal L 
:= \mathcal J \setminus (\mathcal J_{\cap} \cup \mathcal J_s)$
gives the desired lamination,
see 
also the end of \cite[Section 4]{BBD18} for details.
The proof is complete.
\end{proof}

Recall that 
Corollary
\ref{c:pk}
in particular implies 
Theorem \ref{t:main-intro}.

\begin{proof}[Proof of Corollary \ref{c:pk}]
By Theorem \ref{t:mainlamin}, we only need to show that we can take $M_{\lam_0, h}=M$ for all $\lam_0\in M$ and $h> \log d^{k-1}$. This is clear since 
for every endomorphism of $\mathbb P^k$ we have $d^*_{k-1} = d^{k-1}$. The proof is complete.
\end{proof}

\appendix

\section{Exponential backward contraction along graphs}
\label{app:contraction}

Recall that $\mathcal J$ 
is as in \eqref{e:J} and
is a compact metric space, and that $\mathcal F\colon \mathcal J\to \mathcal J$
is defined as 
$(\mathcal F \gamma) (\lam) = f_\lambda (\gamma(\lambda))$.
A web $\mathcal M$
(associated to any $f_\lambda$-invariant measure $\nu$ at any parameter $\lam$)
is an $\mathcal F$-invariant probability measure on $\mathcal J$, and 
$\mathcal M$ is acritical if $\mathcal M (\mathcal J_s)=0$, see Definition \ref{d:acritical}.
In particular,  $\mathcal F$ is surjective on $\mathcal X := \mathcal J \setminus \mathcal J_s$. The \emph{natural extension}
$(\hat {\mathcal X}, \hat {\mathcal F}, \hat {\mathcal M})$
of the system $(\mathcal J, \mathcal F, \mathcal M)$ can be 
defined as follows (see for instance \cite[Section 10.4]{CFS}). 
An element $\hat \gamma \in \hat {\mathcal X}$ is a bi-infinite sequence $\hat \gamma := ( \ldots, \gamma_{-1}, \gamma_0, \gamma_1, \ldots)$
of elements of $\mathcal X$ with the property that $\mathcal F (\gamma_j) = \gamma_{j+1}$. 
For $j \in \mathbb Z$,
we denote by $\pi_j \colon \hat {\mathcal X}\to \mathcal X$ the projection
$\hat \gamma \mapsto \gamma_j$.
We also 
denote by $\hat {\mathcal F}$ the shift map on $\hat {\mathcal X}$, i.e.,
for a $\hat \gamma$ as above we set
\[\hat{\mathcal F} (\hat \gamma) :=
(\ldots, \mathcal F (\gamma_{-1}), \mathcal F (\gamma), \mathcal F (\gamma_1), \ldots)=
(\ldots, \gamma_0, \gamma_1, \gamma_2, \ldots).
\]
The map $\hat {\mathcal F}$ is invertible and satisfies
$\pi_j \circ \hat {\mathcal  F} = \mathcal F \circ \pi_j$ for all $j \in \mathbb Z$. 
There exists a probability measure $\hat {\mathcal M}$ on $\hat {\mathcal X}$ such that $(\pi_j)_* \hat {\mathcal M}= \mathcal M$ for all $j\in \mathbb Z$. This measure is ergodic if $\mathcal M$ is ergodic. 

\medskip

The graph of any element $\gamma \in \mathcal X$ does not intersect the (graph of the) critical orbit of the family $(f_\lambda)_{\lambda\in M}$. 
It follows that, for every $\gamma \in \mathcal X$,
the inverse branches of the holomorphic map
$(\lambda, f_{\lambda})$ are well-defined in a open neighbourhood of the graph of $\gamma$. Given $\hat \gamma \in \hat {\mathcal X}$, we will denote by
 $f_{\hat \gamma}^{-n}$ the inverse branch of order $n$, defined on an open set as above, sending the graph of $\gamma$ to the graph of $\gamma_{-n}$.
The following proposition,
proved in \cite[Propositions 4.2 and 4.3]{BBD18} in the case of the measure of maximal entropy, 
gives a uniform control on the size of the neighbourhoods where the inverse branches as above are defined, and an explicit control on their contraction.
Given $\gamma\in \mathcal X$, $\eta>0$, and a subset $\Omega\subset M$,
we denote by
$T_{\Omega} (\gamma, \eta)$ the $\eta$-neighbourhood of the graph of $\gamma$ over $\Omega$, i.e., we set
\[
T_{\Omega} (\gamma, \eta) := \{
(\lambda, z) \in \Omega \times \C^k\colon |z-\gamma(\lambda)|< \eta
\}.
\]

\begin{proposition}\label{p:bbd-gen}
  Let
  $M$ be a connected and simply connected complex manifold and
$(f_\lambda)_{\lambda \in M}$
 a holomorphic family of polynomial-like
maps of large topological degree.
 Assume that there exists 
  a constant 
  $A_1>0$ and an
   ergodic 
  acritical 
  web
 $\mathcal M$ 
 with the property that
  the Lyapunov exponents of
 $(p_\lam)_* \mathcal M$ 
 are strictly larger than $A_1$ for all $\lam \in M$.
  Then, for every
 open subset $\Omega \Subset M$
 and constant $0<A<A_1$,
 there exists $p\geq 1$, a Borel subset $\hat {\mathcal Y}\subseteq \hat {\mathcal X}$ with
 $\hat {\mathcal M} (\hat { \mathcal Y})=1$,
 and two measurable functions
 $\hat \eta_{p,A} \colon \hat {\mathcal Y} \to ]0,1]$
 and 
 $\hat l_{p,A} \colon \hat {\mathcal Y} \to [1,+\infty[$ which satisfy the following properties.

 For every $\hat \gamma \in \mathcal Y$ and every $n\in p \mathbb N^*$ the iterated inverse branch $f_{\hat \gamma}^{-n}$ is defined on the tubular neighbourhood 
 $T_{\Omega} (\gamma_0, \hat \eta_{p,A} (\hat \gamma))$
 of the graph $\Gamma_{\gamma_0}\cap (\Omega \times \C^k)$ of $\gamma_0$,
 and we have
 \[
f_{\hat \gamma}^{-n} (T_{\Omega} (\gamma_0, \hat \eta_{p,A} (\hat \gamma) ))\subset T_{\Omega} (\gamma_{-n}, e^{-nA})
\quad \mbox{ and }
\quad
\widetilde \Lip ( f^{-n}_{\hat \gamma} )\leq \hat l_{p,A}  (\hat \gamma) e^{-nA},
\]
where $\widetilde \Lip (f_{\hat \gamma}^{-n}):=\sup_{\lam \in \Omega}\Lip  (  (f^{-n}_{\hat \gamma})_{|B(\gamma_0(\lam), \hat \eta_{p,A}) })$.
\end{proposition}

 Observe that, without loss of generality, one can actually assume that $p=1$ in the above statement.

\bibliographystyle{alpha}

\end{document}